\documentclass[12pt]{amsart}
\usepackage{tikz-cd}
\usepackage{amsmath}
\usepackage{fourier}
\usepackage{amssymb}
\usepackage{amscd}
\usepackage{amsthm}
\usepackage[centertags]{amsmath}
\usepackage{amsfonts}
\usepackage{newlfont}
\usepackage{graphicx}
\usepackage{amsfonts, amssymb}
\usepackage{mathrsfs}
\usepackage{latexsym}
\usepackage{tikz}
\usepackage{verbatim}
\usepackage[all]{xy}
\usepackage{enumitem}
\usepackage[colorlinks=true,linkcolor=colorref,citecolor=colorcita,urlcolor=colorweb]{hyperref}
\definecolor{colorcita}{RGB}{21,86,130}
\definecolor{colorref}{RGB}{5,10,177}
\definecolor{colorweb}{RGB}{177,6,38}

\numberwithin{subsection}{section}

\newtheorem{theorem}{Theorem}[section]

\newtheorem{proposition}[theorem]{Proposition}
\newtheorem{corollary}[theorem]{Corollary}

\theoremstyle{definition}
\newtheorem{remark}[theorem]{Remark}
\newtheorem{example}[theorem]{Example}

\theoremstyle{remark}

\newcommand{\dis}{\displaystyle}

\usepackage{mathtools}

\DeclareMathOperator{\mon}{mon}

\DeclareMathOperator{\re}{Re}

\DeclareMathOperator{\id}{\mathrm{id}}

\newcommand{\TT}{\mathbb T}

\newcommand{\chimon}{\chi_{\mon}}

\begin{document}
\title{Projection constants for spaces of Dirichlet polynomials}

\author[Defant]{A.~Defant}
\address{%
Institut f\"{u}r Mathematik,
Carl von Ossietzky Universit\"at,
26111 Oldenburg,
Germany}
\email{defant@mathematik.uni-oldenburg.de}

\author[Galicer]{D.~Galicer}
\address{Departamento de Matem\'{a}tica,
Facultad de Cs. Exactas y Naturales, Universidad de Buenos Aires and IMAS-CONICET. Ciudad Universitaria, Pabell\'on I
(C1428EGA) C.A.B.A., Argentina}
\email{dgalicer$@$dm.uba.ar}

\author[Mansilla]{M.~Mansilla}
\address{Departamento de Matem\'{a}tica,
Facultad de Cs. Exactas y Naturales, Universidad de Buenos Aires and IAM-CONICET. Saavedra 15 (C1083ACA) C.A.B.A., Argentina}
\email{mmansilla$@$dm.uba.ar}

\author[Masty{\l}o]{M.~Masty{\l}o}
\address{Faculty of Mathematics and Computer Science, Adam Mickiewicz University, Pozna{\'n}, Uniwersytetu \linebreak 
 Pozna{\'n}skiego 4,
61-614 Pozna{\'n}, Poland}
\email{mieczyslaw.mastylo$@$amu.edu.pl}

\author[Muro]{S.~Muro}
\address{FCEIA, Universidad Nacional de Rosario and CIFASIS, CONICET, Ocampo $\&$ Esmeralda, S2000 Rosario, Argentina}
\email{muro$@$cifasis-conicet.gov.ar}

\date{}

\thanks{The research of the fourth author was supported by the National Science Centre (NCN), Poland,
Project 2019/33/B/ST1/00165; the second, third and fifth author were supported by CONICET-PIP 11220200102336
and PICT 2018-4250. The research of the fifth author is additionally supported by ING-586-UNR}

\begin{abstract}
\noindent
Given a frequency sequence $\omega=(\omega_n)$ and a finite subset  $J \subset \mathbb{N}$, we study
the space $\mathcal{H}_{\infty}^{J}(\omega)$ of all Dirichlet polynomials
$D(s) := \sum_{n \in J} a_n e^{-\omega_n s}, \, s \in \mathbb{C}$. The main aim is to prove asymptotically correct estimates for the
projection constant $\boldsymbol{\lambda}\big(\mathcal{H}_\infty^{J}(\omega) \big)$ of the finite
dimensional Banach space $\mathcal{H}_\infty^{J}(\omega)$ equipped with the norm
$\|D\|= \sup_{\text{Re}\,s>0} |D(s)|$. Based on harmonic analysis on
$\omega$-Dirichlet groups, we prove
the formula $
\boldsymbol{\lambda}\big(\mathcal{H}_\infty^{J}(\omega) \big) ~
= ~ \dis\lim_{T \to \infty} \frac{1}{2T} \int_{-T}^T \Big|\sum_{n \in J} e^{-i\omega_n t}\Big|\,dt\,,
$
and apply it   to various concrete frequencies $\omega$ and index sets $J$. To see an example, combining
 with a~recent deep result of Harper from probabilistic analytic number theory, we for the space
$\mathcal{H}_\infty^{\leq x}\big( (\log n)\big)$ of all ordinary Dirichlet polynomials
$D(s) = \sum_{n \leq x} a_n n^{-s}$ of length $x$ show   the  asymptotically correct order
$
\boldsymbol{\lambda}\big(\mathcal{H}_\infty^{\leq x}\big( (\log n)\big)\big)
\sim \sqrt{x}/(\log \log x)^{\frac{1}{4}}.
$
\end{abstract}


\maketitle

%\tableofcontents

\section*{Introduction}
\label{Introduction}

The study of complemented subspaces of a Banach space and their projection constants  has a~long history going  back to the beginning of operator theory in Banach spaces. Recall that if $X$ is 
a~closed subspace of a~Banach space $Y$, then the relative projection constant of $X$ in $Y$ is defined by

\[
\boldsymbol{\lambda}(X, Y) = \inf\big\{\|P\|: \,\, P\in \mathcal{L}(Y, X),\,\, P|_{X} = \id_X\big\}\,,
\]
where $\id_X$ is the identity operator on $X$ and as usual $\mathcal{L}(U, V)$ denotes the Banach space of all bounded linear operators between the Banach spaces $U$ and $V$ with the uniform norm. We use here the convention that $\inf\emptyset = \infty$.

The following straightforward result shows the intimate link  between projection constants and extensions of linear operators:
For every Banach space $Y$ and its subspace $X$ one has
\begin{align*}
\boldsymbol{\lambda}(X, Y) = \inf\big\{c>0: \,\,\text{$\forall\, T \in \mathcal{L}(X, Z)$ \,\, $\exists$\,\, an extension\,
$\widetilde{T}\in \mathcal{L}(Y, Z)$\, with $\|\widetilde{T}\| \leq c\,\|T\|$}\big\}\,,
\end{align*}
where $Z$ is any Banach space.
Moreover, the (absolute) projection constant of a Banach space $X$ is given by
\[
\boldsymbol{\lambda}(X) := \sup \,\,\boldsymbol{\lambda}(I(X),Y)\,,
\]
where the supremum is taken over all Banach spaces $Y$ for which it exists some isometric embedding $I\colon X \to Y$ such that $I(X)$ is complemented in $Y$.

A frequency  $\omega = (\omega_n)_{n\in \mathbb{N}}$  is a~strictly increasing,
non-negative real sequence such that $\omega_n \to \infty$ as $n\to \infty$. Given a~finite index set $J \subset \mathbb{N}$
and complex numbers $(a_n)_{n\in J}$,  we say that
\[
D(s) := \sum_{n \in J} a_n e^{-\omega_n s}, \quad\, s \in \mathbb{C}
\]
is a $\omega$-Dirichlet polynomial supported on  the index set $J$. For the frequency $\omega = (n)_{n \in \mathbb{N}}$
we obtain (after the substitution $z = e^{-s}$) polynomials $\sum_{n \in J} a_n z^n$ in one complex variable, and in
the case $\omega = (\log n)_{n \in \mathbb{N}}$ all ordinary Dirichlet polynomials $\sum_{n \in J} a_n n^{-s}$.

Denote by $\mathcal{H}_\infty^{J}(\omega)$ the (finite dimensional) Banach space of all $\omega$-Dirichlet polynomials
supported on the finite index subset  $J \subset \mathbb{N}$, endowed with the norm
\begin{equation*}\label{norm-di}
\|D\|_\infty :=  \sup_{t \in \mathbb{R}} \Big|\sum_{n \in J} a_n e^{-i\omega_n t}\Big|
= \sup_{\re s >0 } \Big|\sum_{n \in J} a_n e^{-\omega_n s}\Big|\,,
\end{equation*}
where the last equality is a simple consequence of the maximum modulus principle.

Then the main goal of this article
is to study the projection constant $$\boldsymbol{\lambda}\big(\mathcal{H}_\infty^{J}(\omega)\big)$$ for various 'natural'
frequencies $\omega$ and various 'natural' finite index sets  $J$ of $\mathbb{N}$. Given $x \in \mathbb{N}$, we are
particularly interested  in the projection constant of the Banach space
\[
\mathcal{H}_\infty^{\leq x}(\omega) = \mathcal{H}_\infty^{\{n\in \mathbb{N}\colon n \leq x\}}(\omega)\,,
\]
so all $\omega$-Dirichlet polynomials $D(s) = \sum_{n \leq x} a_n e^{-\omega_n s}$ of length $x$.

Before we illustrate some of our main results, let us pause for a~moment to say a~few words about the modern theory of
Dirichlet series. Within the last two decades, the theory of ordinary Dirichlet series $\sum a_{n} n^{-s}$ experienced
a~kind of renaissance. The study of these series in fact was one of the hot topics in mathematics at the beginning of
the 20th century.  Among others, H.~Bohr, Besicovitch, Bohnenblust, Hardy, Hille, Landau, Perron, and M.~Riesz.

However, this research took place before the modern interplay between function theory and functional analysis, as well as the
advent of the field of several complex variables, and the area was in many ways dormant until the late 1990s. One of the main
goals of the 1997 paper of Hedenmalm, Lindqvist, and Seip \cite{hedenmalm1997hilbert} was to initiate a~systematic study of
Dirichlet series from the point of view of modern operator-related function theory and harmonic analysis. Independently, at the
same time, a~paper of Boas and Khavinson \cite{bohnenblust1931absolute} attracted renewed attention, in the context of several
complex variables, to the original work of Bohr.

A new field emerged intertwining the classical work in novel ways with modern functional analysis, infinite dimensional holomorphy,
probability theory as well as analytic number theory. As a~consequence, a~number of  challenging research problems crystallized
and were solved over the last decades. We refer to the monographs \cite{defant2019libro}, \cite{HelsonBook}, and
\cite{queffelec2013diophantine}, where many   of the   key elements of this new developments for ordinary Dirichlet series are
described in detail.

Contemporary research in this field owes much to the following fundamental observation of H.~Bohr \cite{bohr1913ueber}, sometimes
called Bohr's vision:

By the transformation $z_j=\mathfrak{p}_j^{-s}$  and the fundamental theorem of arithmetics, an ordinary Dirichlet series
$\sum a_{n} n^{-s}$  may be thought of as a~function $\sum_{\alpha \in \mathbb{N}_0^{(\mathbb{N})}} a_{\mathfrak{p}^\alpha} z^{\alpha}$
of infinitely many complex variables $z_1, z_2, \ldots $, where $\mathfrak{p} =(\mathfrak{p}_j)$ stands for the sequence of prime
numbers. By a~classical approximation theorem of Kronecker, this is more than just a~formal transformation: If, say, only a finite number
of the coefficients $a_n$ are nonzero (so that questions about convergence of the series are avoided), the supremum of the Dirichlet
polynomial $\sum a_n n^{-s}$ in the half-plane $\operatorname{Re} s>0$ equals the supremum of the corresponding polynomial on the
infinite-dimensional circle group  $\mathbb{T}^\infty$. Notice that a Dirichlet polynomial 
$ \sum_{n =1}^x a_n n^{-s}  $  
corresponds to a polynomial 
$ \sum_{\alpha \in \mathbb{N}_0^{(\mathbb{N})}} a_{\mathfrak{p}^\alpha} z^{\alpha} $  on the finite dimensional polytorus $\TT^{\pi(x)}$, where $\pi(x)$ is the amount of prime numbers less or equal $x$. 
Thus, the supremum can be taken over $\TT^{\infty}$ or $\TT^{\pi(x)}$ indistinctly.

Let us sketch some of our main results, as well as some ideas for the strategies how to derive them.

Bohr's vision in its original form is the  seed which allows  to associate with each frequency $\omega$ a~compact
abelian group $G$ (a so-called $\omega$-Dirichlet group) as well as a sequence of characters $h_{\omega_n}$ on $G$,
such that for each finite set $J \subset \mathbb{N}$  the mapping
$
\sum_J a_n e^{-\omega_n s} \mapsto \sum_J a_n h_{\omega_n}
$
leads to a~coefficient preserving identification of $\mathcal{H}_\infty^{J}(\omega)$ with the Banach space
$\text{Trig}_J(G)$ of all trigonometric polynomials on $G$ having Fourier coefficients supported in
$\{h_{\omega_n} \colon n \in J\}$ (see Section~\ref{Bohr's vision}).

Using a famous averaging technique of Rudin (Theorem~\ref{C(G)proj} and Theorem~\ref{rudy}), we derive
the following integral formula (see Theorem~\ref{main-dirB})
\begin{align*}
\boldsymbol{\lambda}\big(\mathcal{H}_\infty^{J}(\omega)\big)
=\int_G \Big|  \sum_{n \in J} h_{\omega_n}\Big|\,d\mathrm{m}\,,
\end{align*}
where $\mathrm{m}$ stands for the Haar measure on $G$, and then equivalently 
\begin{align*}
\boldsymbol{\lambda}\big(\mathcal{H}_\infty^{J}(\omega)\big)
=\lim_{T \to \infty} \frac{1}{2T} \int_{-T}^T \Big|\sum_{n \in J} e^{-i\omega_n t}\Big|\,dt
\end{align*}
(Theorem~\ref{main-dir}).
Our main applications for concrete frequencies $\omega$ and  concrete index sets $J$ offer concrete  estimates
for $\boldsymbol{\lambda}\big(\mathcal{H}_\infty^{J}(\omega)\big)$, and their proofs are all mainly based on caculating one of the  preceding two integrals.

For the three standard frequencies
$\omega = (n)_{n \in \mathbb{N}_0}$, $\omega = (\log p_n)_{n \in \mathbb{N}}$ (where as above $p_n$ is the $n$th prime number), and $\omega = (\log n)_{n \in \mathbb{N}}$, we in Section~\ref{Projection constants II}, Case I, get the following formulas which are asymptotically correct in $x$:

\begin{itemize}
\item[$\bullet$]
$\boldsymbol{\lambda}\big(\mathcal{H}_\infty^{\leq x}\big((n)\big)\big)=  \frac{4}{\pi^2} \log(x+1) + o(1)$\,,
\vspace{2mm}
\item[$\bullet$]
$\lim_{x \to \infty}  \frac{\boldsymbol{\lambda}\big(\mathcal{H}_\infty^{\leq x}\big( (\log p_n)\big)\big)}{\sqrt{x}}
= \frac{\sqrt{\pi}}{2}$\,,
\vspace{2mm}
\item[$\bullet$]
$\boldsymbol{\lambda}\big(\mathcal{H}_\infty^{\leq x}\big((\log n)\big)\big) = O\left(\frac{\sqrt{x}}{(\log \log x)^{\frac{1}{4}}}\right)$\,.
\end{itemize}
To estimate the integrals for the first two statements is fairly standard -- but for the third one this is highly non-trivial.
Indeed, a~recent deep theorem from probabilistic analytic number theory due to Harper from \cite{Harper} shows that
\begin{equation}\label{harperA}
\lim_{T \to \infty }\frac{1}{2T} \int_{-T}^T \Big|\sum_{n=1}^x \frac{1}{n^{it}}\Big| dt  =
O\left(\frac{\sqrt{x}}{(\log \log x)^{\frac{1}{4}}}\right)\,,
\end{equation}
which  is equivalent to
\begin{equation*}\label{harperB}
\int_{\mathbb{T}^\infty} \Bigg| \sum_{\alpha \in \mathbb{N}^{(\mathbb{N})}_0: 1 \leq n \leq x} z^\alpha\Bigg|\,dz
= O\left(\frac{\sqrt{x}}{(\log \log x)^{\frac{1}{4}}}\right)\,.
\end{equation*}
This resolved a~long-standing problem of Helson from \cite{helson2010hankel}. In fact, Helson had conjectured that the
integral is  of order $o(\sqrt{x})$ which would have disproved a~certain generalisation of Nehari's theorem from harmonic
analysis. We also mention that Harper's result  gave a~negative answer to the so-called embedding problem showing that
for $0~<p~<~2$ the $L_p$-integral of every ordinary Dirichlet polynomial $D =\sum_{n \leq x} a_n n^{-s}$ over any segment
of fixed length on the vertical line $[\re s = 1/2]$ can not be bounded by a~universal constant times
$\lim_{T \to \infty }\frac{1}{2T} \int_{-T}^T |D(it)|^p dt$ (see also Problem 2.1 in \cite{saksman2016some}).

Finally, we mention that we also study  the projection constant of Banach spaces of ordinary Dirichlet polynomials supported
on an  index sets of natural numbers with a certain complexity of their prime number decompositions (see again
Section~\ref{Projection constants II}, Case II and III).

We finish the article comparing and linking the results we obtained for  projection constants, with  some important estimates of the unconditional basis constant  $\boldsymbol{\chimon} (\mathcal{H}_{\infty}^{J}(\omega))$ (see Section \ref{sidon} for the definition), for several natural frequencies  $\omega=(\omega_n)$ and index sets  $J \subset \mathbb{N}$. From the group point of view, this is related to the Sidon constant of  specific sets of characters.

\section{Preliminaries}

We use standard notation from Banach space theory as e.g., used in the monographs \cite{diestel1995absolutely,LT1,
pisier1986factorization, tomczak1989banach, wojtaszczyk1996banach}. If not indicated differently, we consider complex
Banach spaces.

Given two sequences $(a_n)$ and $(b_n)$ of non-negative real numbers we write $a_n \prec b_n$, if there is a~constant
$c>0$ such that $a_n \leq c\,b_n$ for all $n\in \mathbb{N}$, while $a_n \sim b_n$ means that $a_n \prec b_n$ and
$b_n \prec a_n$ holds. In the case that an extra parameter $m$ is also involved, for two sequences of non-negative real
numbers $(a_{n,m})$ and $(b_{n,m})$, we write $a_{n,m} \prec_{C(m)} b_{n,m}$ whenever there is a constant $C(m)>0$
(which depends exclusively on $m$ and not on $n$) such that $a_{n,m} \leq C(m) b_{n,m}$ for all $n,m \in \mathbb{N}$.
We use the notation $a_{n,m} \sim_{C(m)} b_{n,m}$ if $a_{n,m} \prec_{C(m)} b_{n,m}$ and $b_{n,m} \succ_{C(m)} a_{n,m}$.
We also write $a_{n,m} \prec_{C^m} b_{n,m}$ when there is a hypercontractive comparision, i.e., there is an absolute
constant $C>0$ such that $a_{n,m} \leq C^m b_{n,m}$ for all $n,m \in \mathbb{N}$. Of course, if $a_{n,m} \prec_{C^m}
b_{n,m}$ and $b_{n,m} \succ_{C^m} a_{n,m}$ we simply write $a_{n,m} \sim_{C^m} b_{n,m}$.

In the following two paragraphs we collect a few standard facts on projection constants as well as topological groups.

\subsection{Projection constants} \label{Projection constants}
Any Banach space $X$ can be embedded isometrically into $\ell_\infty(S)$, where  $S$ is a~nonempty set which in general
depends on $X$.
Indeed, if $S$ is the unit ball of the dual of $X$, then  it follows from the Hahn–Banach theorem that the mapping $x \mapsto (\varphi(x))_{\varphi \in S}$
 is an isometric embedding from $X$ into $\ell_{\infty}(S)$.
Moreover, for  every separable $X$, we may choose $S=\mathbb{N}$ (see e.g. \cite[Theorem 2.5.7]{albiac2006topics}). Throughout the paper, we  use the fact that, if $S$
is a~nonempty set for which the Banach space $X$ is  isometrically isomorphic to a~subspace $Z$ of $\ell_{\infty}(S)$, then
$
\boldsymbol{\lambda}(X) = \boldsymbol{\lambda}(Z, \ell_\infty(S))\,.
$
Indeed, this is due to the fact that $\ell_{\infty}(S)$ is  isometrically injective (see \cite[Definition 2.5.1. and Proposition 2.5.2.]{albiac2006topics}).
Thus finding $\boldsymbol{\lambda}(X)$ is equivalent to finding the norm of a~minimal projection from $\ell_\infty(S)$ onto an isometric copy of
$X$ in $\ell_\infty(S)$. However, this is a~non-trivial problem in general.

General bounds for projection constants of various finite dimensional Banach spaces were studied by many authors. The most fundamental general
upper bound is due to Kadets and Snobar \cite{kadecsnobar}: For every $n$-dimensional Banach space $X_n$ one has
\begin{equation} \label{kadets1}
\boldsymbol{\lambda}(X_n) \leq \sqrt{n}\,.
\end{equation}
In contrast, K\"onig and Lewis \cite{koniglewis} showed that for any Banach space $X_n$ of dimension $n \ge 2$ the strict
inequality $\boldsymbol{\lambda}(X_n) < \sqrt{n}$ holds, and this estimate was improved by Lewis \cite{lewis} showing
\[
\boldsymbol{\lambda}(X_n) \leq \sqrt{n}\,\left(1 - \frac{1}{n^2} \bigg(\frac{1}{5}\bigg)^{2n + 11}\right)\,.
\]
The exact values of $\boldsymbol{\lambda}(\ell_2^n)$ and $\boldsymbol{\lambda}(\ell_1^n)$ were computed by Gr\"unbaum
\cite{grunbaum} and Rutovitz \cite{rutovitz}: In the complex case
\begin{align}\label{grunbuschC-A}
\boldsymbol{\lambda}\big(\ell_2^n(\mathbb{C})\big)  = n \int_{\mathbb{S}_n(\mathbb{C})} |x_1|\,d\sigma
= \frac{\sqrt{\pi}}{2}   \frac{n!}{\Gamma(n + \frac{1}{2})}\,,
\end{align}
where $d\sigma$ stands for the normalized surface measure on the sphere  $\mathbb{S}_n(\mathbb{C})$
in $\mathbb{C}^n$, and
\begin{align}\label{grunbuschC-B}
\boldsymbol{\lambda}\big(\ell_1^n(\mathbb{C})\big)  = \int_{\mathbb{T}^n} \Big|\sum_{k=1}^{n} z_k\Big|\, dz
= \int_{0}^{\infty} \frac{1 -J_0(t)^n}{t^2} dt\,,
\end{align}
where $dz$ denotes the normalized Lebesgue measure on the distinguished boundary $\mathbb{T}^n$ in $\mathbb{C}^n$
and $J_0$ is the zero Bessel function defined by
$
J_0(t) = \frac{1}{2\pi} \int_{0}^{\infty} \cos( t \cos \varphi) d \varphi\,.
$
The corresponding real constants are different:
\begin{align}\label{grunbuschR}
& \boldsymbol{\lambda}\big(\ell_2^n(\mathbb{R})\big)  =  n \int_{\mathbb{S}_n(\mathbb{R})} |x_1|\,d\sigma
= \frac{2}{\sqrt{\pi}}   \frac{\Gamma(\frac{n+2}{2})}{\Gamma(\frac{n+1}{2})} \\
& \boldsymbol{\lambda}\big(\ell_1^n(\mathbb{R})
\big)  =
\begin{cases}
\boldsymbol{\lambda}\big(\ell_2^n(\mathbb{R})\big),  &  \text{$n$ odd}\\[2mm]
\boldsymbol{\lambda}(\ell_2^{n-1}(\mathbb{R})),  &  \text{$n$ even}\,.\\[2mm]
\end{cases}
\end{align}
Gordon \cite{gordon} and Garling-Gordon \cite{garlinggordon} determined the asymptotic growth of
$\boldsymbol{\lambda}\big(\ell_p^n\big)$ for $1<p<\infty$ with $p \notin \{1, 2, \infty\}$:
\begin{equation} \label{gordongarling}
\boldsymbol{\lambda}\big(\ell^n_{p}\big) \sim n^{\min\big\{\frac{1}{2}, \frac{1}{p} \big\}}\,.
\end{equation}
K\"onig, Sch\"utt and Tomczak-Jagermann \cite{konig1999projection} proved that for  $1 \le p \le 2$
\begin{align}\label{koenigschuetttomczak}
  \lim_{n\to \infty} \frac{\boldsymbol{\lambda}\big(\ell_p^n\big)}{\sqrt{n}} = \gamma\,,
\end{align}
where $\gamma = \sqrt{\frac{2}{\pi}}$ in the real  and $\gamma= \frac{\sqrt{\pi}}{2}$ in the complex case.
 For an extensive treatment on all of this and more see the excellent  monograph
\cite{tomczak1989banach}
of Tomczak-Jaegermann.

\subsection{Topological groups} \label{Topological groups}
 As usual a~group $G$ equipped with a~topology $\tau$ is said to be
a~topological group whenever the mapping $(G, \tau) \times (G, \tau) \ni (a, b) \to a b^{-1} \in (G, \tau)$ is continuous.
From here on $G$ is assumed to be a compact group, that is, its topology is compact. In this case, $G$ defines a~natural set of maps
$\{L_a\}_{a\in G}$ and $\{R_a\}_{a\in G}$ on  $C(G)$, the complex-valued continuous functions on $G$, given
for all $a, b\in G$ by
\[
L_af(b):= f(ab), \quad\, \, \, \,\text{and\, $R_af(b) = f(ba)$}, \quad\, f\in C(G)\,.
\]
It is well-known that for every compact group $G$ there exists a~unique Borel~probability measure $\mathrm{m}$ which is left
invariant, that is,
\[
\int_G f(b)\,d\mathrm{m}(b) = \int_G L_af(b)\,d\mathrm{m}(b), \quad\, a\in G, \, \, \, f\in C(G)\,.
\]
This $\mathrm{m}$ is called the Haar measure of $G$. If in addition $\mathrm{m}$ is also right invariant:
\[
\int_G f(b)\,d\mathrm{m}(b) = \int_G R_af(b)\,d\mathrm{m}(b), \quad\, a\in G, \, \, \, f\in C(G)\,,
\]
then the compact group $G$ is called unimodular. Examples of unimodular groups are compact groups in which every one point set is closed.

Let $\mathrm{m}$ be the normalized Haar measure on $G$, and $\widehat{G}$ as usual the dual group of $G$ (i.e., the set of all continuous characters on $G$). For any $f\in L^1(G):=L^1(G, \mathrm{m})$, the Fourier transform of $f$
is given by
\[
\widehat{f}(\gamma):= \int_G\, f(a)\,\overline{\gamma(a)} \,d\mathrm{m}(a), \quad\, \gamma \in \widehat{G}\,.
\]

Recall that $L_1(G)$ forms a~commutative Banach algebra, whenever it carries  the convolution $f_1\ast f_2$ as its multiplication, that is, for $\mathrm{m}$-almost every $a\in G$
\[
(f_1\ast f_2)(a) := \int_G f_1(ab^{-1})f_2(b)\,d\mathrm{m}(b)\,.
\]

Products of groups are going to be of particular interest for our purposes.
 Given compact abelian groups $G_1, \ldots, G_n$, each with the Haar measure $\mathrm{m}_j, 1\leq j\leq n$, we denote  by
$G:= G_1 \times \cdots \times G_n$ the product of these groups endowed  its natural  product operation and product topology . Given an $n$-tuple
of characters $(\gamma_1, \ldots, \gamma_n) \in \widehat{G_1} \times \cdots \times \widehat{G_n}$ and
$\alpha = (\alpha_1, \ldots, \alpha_n)\in \mathbb{Z}^n$, we write $\gamma^\alpha$ for the~character in $\widehat{G}$
given by
\[
\gamma^\alpha(x_1,\ldots, x_n) := \gamma_1(x_1)^{\alpha_1}\cdots \gamma_n(x_n)^{\alpha_n}, \quad\, (x_1, \ldots, x_n)\in G\,.
\]
Of special interest is the $n$-dimensional circle group
$G:= \mathbb{T}^n$, where the  Haar measure $\mathrm{m}=:dz$ on $\mathbb{T}^n$ acts on a~Borel function $f\colon \mathbb{T}^n \to \mathbb{C} $ by the formula
\[
\int_{\mathbb{T}^n} f(z)\,dz = \frac{1}{(2\pi)^n} \int_0^{2\pi}\cdots \int_0^{2\pi} f(e^{it_1}, \ldots, e^{it_n})\,dt_1\ldots dt_n\,.
\]
Recall that $\widehat{\mathbb{T}^n} = \mathbb{Z}^n$, where the identification is given by the fact that for every character
$\gamma \in \widehat{\mathbb{T}^n} $ there is a unique multi index $(\alpha_1, \ldots, \alpha_n) \in \mathbb{Z}^n$ for which
$\gamma(z) = z^{\alpha}$ for every $z= (z_1, \ldots, z_n)\in \mathbb{T}^n$.

\bigskip

\section{Trigonometric polynomials} \label{Integral formula - trigonometric polynomials}
\label{Integral formula - trigonometric polynomials}

For any compact group $G$ and any  nonempty finite set $E\subset \widehat{G}$ we denote $\text{span}\,E$  by $\text{Trig}_E(G)\,,$ the space of all trigonometric polynomials $P$ on $G$ which are supported on $E$, that is, the Fourier transform $\widehat{P}$ is supported on $E$. In what follows, we consider this finite dimensional space with the supremum norm on $G$. Note that every polynomial $P\in \text{Trig}_E(G)\,$ has the form
\[
P(g) \,= \,\sum_{\gamma \in E} \widehat{P}(\gamma)\, \gamma(g)\,, \quad g \in G\,.
\]
The space of all trigonometric polynomials on $G$ is denoted by $\text{Trig}(G)$.

We need  two  fundamental consequences of the Peter-Weyl theorem (see, e.g., \cite[Theorem 1.3.3]{queffelec2013diophantine}): For
any compact abelian group $G$, the space $\text{Trig}(G)$ is dense in the space $C(G)$ of complex-valued functions
(see \cite[Theorem 1.3.4]{queffelec2013diophantine}), and  the dual group $\widehat{G}$ is an orthonormal basis of the Hilbert space
$L^2(G, \mathrm{m})$ (see \cite[Theorem 1.3.6]{queffelec2013diophantine}).

For the special case $G = \mathbb{T}^n$ with $\widehat{G} =
\mathbb{Z}^n$ we are particularly interested in the index sets
\begin{equation}\label{index1}
J_p(m,n):= \Big\{\alpha \in \mathbb{Z}^n : \,\, \big(\sum_{j} |\alpha_j|^p\big)^\frac{1}{p} = m\Big\},
\quad\, \Lambda_p(m,n):= \Big\{ \alpha \in \mathbb{N}_0^n : \,\, \big(\sum_{j} \alpha_j^p\big)^\frac{1}{p} = m\Big\}
\end{equation}
as well as
\begin{equation}\label{index2}
J_p(\leq m,n):= \Big\{\alpha \in \mathbb{Z}^n : \,\, \big(\sum_{j} \alpha_j^p\big)^\frac{1}{p} \leq  m\Big\}, \quad\,
\Lambda_p(\leq m,n):= \Big\{ \alpha \in \mathbb{N}_0^n : \,\,\big(\sum_{j} |\alpha_j|^p\big)^\frac{1}{p} \leq  m\Big\}\,,
\end{equation}
where $m,n \in \mathbb{N}$ and  $p\in \{1, 2, \infty\}$. For $p=1$
we write   $|\alpha|=\sum_{j} |\alpha_j|$, usually called the  order of $\alpha$, and  for $p = \infty$ we of course by $\big(\sum_{j} |\alpha_j|^p\big)^{1/p}$ mean $\max_j |\alpha_j|$.

For $m~\in~\mathbb{N}$ we write $\text{Trig}_{m}(\mathbb{T}^n) : = \text{Trig}_{\Lambda_1(m,n)}(\mathbb{T}^n)$
for the Banach space of all analytic
trigonometric polynomials which are $m$-homogeneous, i.e., all polynomials of the form
\[
P(z) = \sum_{\alpha \in \Lambda_1(m,n)} c_\alpha z^{\alpha} \quad \text{for}\,\, z\in \mathbb{T}^n\,.
\]
Similarly, $\text{Trig}_{\leq m}(\mathbb{T}^n)  : = \text{Trig}_{\Lambda_1(\leq m,n)}(\mathbb{T}^n)$,
stands for the space of all analytic trigonometric polynomials of degree $\leq m$.

The following integral formula  for the projection constant of the Banach space of all trigonometric polynomials
on a~given compact abelian group supported on a  priori given set of characters in $\widehat{G}$, is one of the
main sources of what we intend to do.

\begin{theorem}\label{C(G)proj}
Let $G$ be a compact abelian group and  $E:=\{\gamma_1,\ldots, \gamma_N\}\subset \widehat{G}$  a~finite set of
characters. Then $\Pi \colon C(G) \to C(G)$, given by $\Pi f = \sum_{j=1}^N \widehat{f}(\gamma_j) \gamma_j$ for
all $f\in C(G)$, is  the unique projection onto $\text{Trig}_E(G)$ that commutes with the action of the group on
$C(G)$. Moreover, $\Pi $ is minimal{\rm:}
\[
\boldsymbol{\lambda}\big(\text{Trig}_E(G)\big) = \big\| \Pi \colon C(G)\to C(G)\big\|
= \int_G \Big|\sum_{j=1}^N \gamma_j(a)\Big|\,d\mathrm{m}(a)\,.
\]
\end{theorem}

The following section is devoted to the proof of this result.

\subsection{Averaging projections} \label{Rudin's averaging technique}

As mentioned, one of the main tools we intend to use is a method due to Rudin (see the forthcoming Theorem~\ref{rudy}).
Roughly speaking, under certain assumptions, there is a somewhat universal  averaging technique which allows to construct
new projections with additional and somehow better properties from an a~priori given projection.

Given a topological group $G$   and a Banach space $Y$, we need to explain when all elements of $G$  act  as bounded
linear operators on  $Y$. Formally this means that there is a mapping
\[
T\colon G \to \mathcal{L}(Y)\,, \, \,\, \,a \mapsto T_a
\]
such that
\[
T_e = I_Y, \quad\, T_{ab} = T_a T_b, \quad\,  a, \, \, b\in G
\]
and all mappings
\begin{equation}\label{con(i)}
 G\ni a \mapsto T_a(y) \in Y,  \quad\,   \, \, y\in Y
\end{equation}
are continuous. Then $G$ is said to act on $Y$ through $T$ (or simply, $G$ acts on $Y$). If in addition all operators
$T_a, a\in G$ are isometries, then we say that $G$ acts isometrically on $Y$. We say that $S \in \mathcal{L}(Y)$ commutes
with the action of $G$ on $Y$ through $T$ whenever $S$ commutes with  all $T_b,\, b\in G$.

\begin{theorem} \label{rudy}
Let $Y$ be a~Banach space, $X$ a~complemented subspace of\, $Y$, and \,$\mathbf{Q} \colon Y \to Y$ a~projection onto $X$.
Suppose that $G$ is a~compact group with Haar measure $m$, which   acts on $Y$ through $T$ such that $X$ is invariant under
the action of $G$, that is,
\begin{equation} \label{invariant}
T_a(X) \subset X, \quad\,a\in G\,.
\end{equation}
Then $\mathbf{P}\colon Y \to Y$ given by
\begin{equation}\label{equation rudy}
\mathbf{P}(y):= \int_{G} T_{a^{-1}}\mathbf{Q}T_a(y)\,d\mathrm{m}(a), \quad\, y\in Y\,,
\end{equation}
is a~projection onto $X$ which commutes with the action  of $G$ on $Y$, i.e., $T_a\mathbf{P} = \mathbf{P}T_a$ for all
$a\in G$, and satisfies
\[
\|\mathbf{P}\| \,\leq\, \|\mathbf{Q}\|\,\,\,\sup_{a\in G}\|T_a\|^2\,.
\]
Moreover, if there is a unique projection on $Y$ onto $X$ that commutes with
the  action of $G$ on $Y$, and if $G$ acts isometrically  on $Y$, then $\mathbf{P}$ given in the formula~\eqref{equation rudy} is minimal, i.e.,
\[
\boldsymbol{\lambda}(X,Y) = \|\mathbf{P}\|\,.
\]
\end{theorem}

This result (in one way or the other) found various applications in the literature (see, e.g., \cite{konig1995projections,
konig1999projection, schuettkwapien, rieffel2006lipschitz, light1986minimal,  rudin1962projections,  rudin1986new}.
In the case of the circle group, it can be traced back to Faber's   article  \cite{faber1914interpolatorische}. For the sake of completeness
we include a~proof, which is  inspired by \cite{rudin1962projections} (and  also \cite[Theorem III.B.13]{wojtaszczyk1996banach}).

\begin{proof}[Proof of Theorem~\ref{rudy}]
Note first, that, by the Banach-Steinhaus theorem,  $\sup_{a\in G} \|T_a\|<\infty$. Then, given $y \in Y$, the mapping
$a\ni G \mapsto T_{a^{-1}}\mathbf{Q}T_a \in \mathcal{L}(Y)$ is bounded and by \eqref{con(i)}
measurable
(being almost everywhere separably valued and weakly measurable), and hence Bochner integrable.
Consequently, $\mathbf{P}$ defines an operator on $Y$.
 Moreover, from  \eqref{invariant} we deduce
for all $y\in Y$ and for all $x\in X$ that
\[
T_{a^{-1}}\mathbf{Q}T_a(y) \in X, \quad\, T_{a^{-1}}\mathbf{Q}T_a(x) = x\,,
\]
implying that  $\mathbf{P}$ is a~projection from $Y$ onto $X$. The hypothesis that $G$ acts on $Y$ (through $T$) yields
for all $b\in G$
\[
T_{b^{-1}}\mathbf{P}T_b = \int_G T_{b^{-1}}T_{a^{-1}}\mathbf{Q}T_a T_b\,d\mathrm{m}(a)
= \int_G T_{(ab)^{-1}}\mathbf{Q}T_{ab}\,d\mathrm{m}(a) = \mathbf{P}\,,
\]
so $\mathbf{P}$ commutes with the action of $G$ on $Y$.
Since for all $y\in Y$
\[
\|\mathbf{P}y\|_Y \leq \int_G \|T_{a^{-1}}\|\,\|\mathbf{Q}\,\|\|T_a\|\,\|y\|_Y\,d\mathrm{m}(a)
\leq \|\mathbf{Q}\|\,\sup_{a\in G}\|T_a\|^{2}\,\|y\|_Y\,,
\]
the required estimate for  $\|\mathbf{P}\|$ follows.
If moreover there is a unique projection on $Y$ onto $X$ that commutes with
the  action of $G$ on $Y$, and if
$G$ acts isometrically on $X$,
then the projection $\mathbf{P}$ from
~\eqref{equation rudy}
does not depend on
$\mathbf{Q}$ and
$\|\mathbf{P}\| \,\leq\, \|\mathbf{Q}\|$ for all projections $\mathbf{Q}$ on $Y$ onto $X$, i.e.,
$\|\mathbf{P}\| \leq  \boldsymbol{\lambda}(X,Y)$. This completes the proof.
\end{proof}

\begin{proof}[Proof of  Theorem~\ref{C(G)proj}]Note first  that $G$ in a natural way acts on $C(G)$ (in the sense
of Section~\ref{Rudin's averaging technique}), where the action is given by the mapping $T: G \to \mathcal{L}(C(G)), \,a \mapsto T_a$ with
\[
T_af(b):= f(ab), \quad\, f\in C(G), \,\, b\in G\,.
\]
We claim that $\Pi \colon C(G) \to C(G)$ is the unique projection onto $\text{Trig}_E(G)$ that commutes with all translation operators
$T_a$, $a\in G$. To see this, assume that $\bold{Q} \colon C(G) \to  C(G)$ onto $\text{Trig}_E(G)$ is a~projection that commutes with all
translation operators. Then for all $\gamma, \gamma'\in \widehat{G}$ one has
\[
\widehat{T_a \bold{Q}\gamma}(\gamma') = \widehat{\bold{Q} T_a\gamma}(\gamma')\,.
\]
It is easy to check that $\widehat{T_a\bold{Q}\gamma}(\gamma') = \gamma'(a)\, \widehat{\bold{Q}\gamma}(\gamma')$ and
$\widehat{\bold{Q}T_a\gamma}(\gamma')= \gamma(a)\, \widehat{\bold{Q}\gamma}(\gamma')$. In consequence, we get
\[
\gamma'(a) \widehat{\bold{Q}\gamma}(\gamma') = \gamma(a) \widehat{\bold{Q}\gamma}(\gamma'), \quad\, a\in G\,.
\]
This implies that, for all $\gamma, \gamma'\in \widehat{G}$ with $\gamma \neq \gamma'$, we have
$\widehat{\bold{Q}\gamma}(\gamma') = 0$. On the other hand,  the Peter-Weyl theorem  states that $\widehat{G}$ forms an orthonormal basis in the Hilbert space $L^2(G)$, hence
\[
\bold{Q}\gamma = \sum_{\gamma' \in \widehat{G}} \widehat{\bold{Q}\gamma}(\gamma')\,\gamma', \quad\, \gamma\in \widehat{G}\,,
\]
and consequently, for every character $\gamma \in \widehat{G}$ there is a~scalar $c_\gamma$ such that $\bold{Q}\gamma = c_{\gamma}\gamma$.

Since $\bold{Q}$ is a~projection onto $\text{Trig}(G)$, $c_\gamma = 0$ for all $\gamma \in\widehat{G}\setminus E$, and $c_{\gamma} = 1$
for all $\gamma \in E$. In consequence, $\bold{Q}\gamma = \gamma$ for all $\gamma \in E$, and $\bold{Q}\gamma = 0$ for all
$\gamma  \in \widehat{G} \setminus E$. Hence the projection $\bold{Q}$, restricted to the algebra $\text{Trig}(G)$
of all trigonometric polynomials on $G$, has the representation
\begin{align}\label{representation of proj onto Trig}
\bold{Q}f = \sum_{j=1}^N \widehat{f}(\gamma_j) \gamma_j, \quad\, f\in \text{Trig}(G)\,.
\end{align}
Consequently, we conclude from the density of $\text{Trig}(G)$ in  $C(G)$ that the above formula holds
for all $f\in C(G)$. Hence $\bold{Q}=\Pi$. This proves the claim.

Since $\Pi$ is the unique projection onto $\text{Trig}_E(G)$ that commutes with all translation operators, $\Pi$ is fixed by the averaging technique introduced in Theorem~\ref{rudy}. Now observe that $G$ acts isometrically on $C(G)$, i.e., all mappings  $T_a\colon C(G) \to C(G)$ are isometries on  $C(G)$. Since for all $a\in G$
\[
T_a\gamma  = \gamma(a)\gamma, \quad\, \gamma \in \widehat{G}\,,
\]
it follows that $T_a(\text{Trig}_E(G)) \subset \text{Trig}_E(G)$ for all $a\in G$.  Then, by the moreover part of Theorem~\ref{rudy},
$\Pi$ is a~minimal projection, that is,
\[
\boldsymbol{\lambda}\big(\text{Trig}_E(G)\big) = \big\| \Pi \colon C(G)\to C(G)\big\|\,.
\]
Finally, it remains to prove the integral formula for the norm of  $\Pi$. Since $f \ast \gamma=\widehat{f}(\gamma)\gamma$
for all $f\in L^1(G)$ and $\gamma \in \widehat{G}$, we get by  \eqref{representation of proj onto Trig},
\[
\Pi f= f \ast \Big(\sum_{j=1}^N \gamma_j\Big), \quad\, f\in C(G)\,.
\]
Clearly, $\sum_{j=1}^N \gamma_j \in C(G)$, so it can readily be shown by direct computation that
\[
\big\| \Pi \colon C(G) \to C(G)\big\| = \int_G \, \Big|\sum_{j=1}^N \gamma_j(a)\Big|\,d\mathrm{m}(a).
\]
Indeed, we have

\begin{align*}
\big\| \Pi \colon C(G) \to C(G)\big\|  &  =\sup_{\Vert f \Vert_{\infty} =1 } \; \; \; \sup_{b \in G} \; \; \; \left\vert \int_G f(a) \Big(\sum_{j=1}^N \gamma_j\Big)( b  a^{-1}) d\mathrm{m}(a)
\right\vert \\
& = \sup_{b \in G} \; \; \; \sup_{\Vert f \Vert_{\infty} =1 }  \; \; \; \left\vert \int_G f(a) \Big(\sum_{j=1}^N \gamma_j\Big) (b  a^{-1}) d\mathrm{m}(a)
\right\vert \\
& = \sup_{b \in G} \; \; \; \int_G \Big|\sum_{j=1}^N \gamma_j(b  a^{-1}) \Big|\,d\mathrm{m}(a)\,.
\end{align*}
Now note that by the translation invariance of the Haar measure $\mathrm{m}$ and the fact that it is both left and right invariant ($G$ is unimodular since it is abelian), for each $b\in G$ we have,
\[
\int_G \Big|\sum_{j=1}^N \gamma_j(b  a^{-1})\Big|\,d\mathrm{m}(a)    = \int_G \, \Big|\sum_{j=1}^N \gamma_j(a^{-1})\Big|\,d\mathrm{m}(a) =\int_G \, \Big|\sum_{j=1}^N \gamma_j(a)\Big|\,d\mathrm{m}(a)\,.
\]
This completes the proof.
\end{proof}

\subsection{Projection constants - trigonometric polynomials}

Based on the integral formula from Theorem~\ref{C(G)proj}, we derive  concrete formulas and asymptotically optimal
estimates for the projection constants of spaces of trigonometric polynomials on  compact abelian groups, which
have Fourier expansions supported on an a~priori fixed  set of characters.

\subsubsection{$\Lambda(2)$-sets} \label{L2}

In order to show a very first application of Theorem \ref{C(G)proj}, we need some further notation and preliminaries.
Recall that Rudin in his  classical paper \cite{rudin1960trigonometric} from  1960 introduced the notion of $\Lambda(p)$-sets
within the setting of Fourier analysis on the circle group $\mathbb{T}$. In modern language, if $G$ is a~compact abelian group
(with Haar measure $\mathrm{m}$) and $p\in (1, \infty)$, then the~subset ${E} \subset \widehat{G}$ is said to be a~$\Lambda(p)$-set whenever there exists a~constant $C>0$ such that, for every trigonometric polynomial $P \in \text{Trig}_E(G)$, one has
\begin{equation}\label{lambdap}
\|P\|_{L_p(G)} \leq C \|P\|_{L_1(G)}\,.
\end{equation}
In this case, the least such constant  is called the $\Lambda(p)$-constant of $E$, and denoted by $C_p=C_p(E)$. Let us here
remark that for  $p>2$ the validity of \eqref{lambdap}  is equivalent to the existence of a~constant $A_p > 0$ such that
\[
\|P\|_{L_p(G)} \leq A_p \|P\|_{L_2(G)}, \quad\, P\in \text{Trig}_E(G)\,.
\]
The following almost immediate consequence of Theorem \ref{C(G)proj} shows that $\Lambda(2)$-sets are of particular
importance for our purposes -- see also Corollary~\ref{manydimensionsB} and Corollary~\ref{b2dirichelet}.

\begin{corollary} \label{corB2}
Let $G$ be a compact abelian group. Then, for any  finite set $E= \{\gamma_1,\ldots, \gamma_N\}\subset \widehat{G}$
of different characters with $\Lambda(2)$ constant $C_2$, we have
\[
C_2^{-1} \sqrt{N} \leq \boldsymbol{\lambda}\big(\text{Trig}_E(G)\big) \leq \sqrt{N}\,.
\]
\end{corollary}

\begin{proof}
Let $\mathrm{m}$ be the Haar measure on $G$. Then from Theorem \ref{C(G)proj}, we get
\[
\boldsymbol{\lambda}\big(\text{Trig}_E(G)\big)= \int_G |\gamma_1(x) + \ldots + \gamma_N(x)|\,d\mathrm{m}(x)\,.
\]
Since $\widehat{G}$ is an orthonormal basis in $L_2(G, \mathrm{m})$,
\[
\bigg( \int_G |\gamma_1(x) + \ldots + \gamma_N(x)|^2 \, d \mathrm{m}(x)
\bigg)^{\frac{1}{2}} = \sqrt{N} .
\]
The proof is completed by using \eqref{lambdap} and the Cauchy-Schwarz inequality.
\end{proof}

In combination with the preceding corollary we need the following example.

\begin{remark} \label{B2}
Let $G$ be a compact abelian group. Following \cite{grahamhare2013}, a~set $E \subset \widehat{G}$ is said to be
a~$B_2$-set, whenever $\gamma_1 \gamma_2 = \gamma_3 \gamma_4$ for all $\gamma_1, \ldots, \gamma_4 \in E$ if and only
if $\{\gamma_3, \gamma_4\}$ is a permutation of $\{\gamma_1, \gamma_2\}$. It is worth noting that $B_2$-sets are
$\Lambda(2)$-sets with $\Lambda(2)$-constant $\sqrt{2}$ (see
\cite[Proposition 6.3.11]{grahamhare2013}).

In passing we observe that \cite[Proposition 6.3.11]{grahamhare2013}, also implies the following: If $E\subset \widehat{G}$
fulfills that there exists $N\ge 2$ such that for every $\gamma\in \widehat{G}$ there are at most $N$ pairs of the form
$(\gamma_{i_1},\gamma_{i_2})\subset E\times E$ with $\gamma_{i_1}\gamma_{i_2}=\gamma$, then  $E$ is $\Lambda(2)$-set
with $C_2\leq \sqrt{N}$.
\end{remark}

Another example of  $\Lambda(2)$-sets, important for our purposes, follows from an inequality due to  Weissler \cite{weissler1980logarithmic}
(see also \cite[Theorem~8.10]{defant2019libro}):
For all $P \in \text{Trig}_{\leq m}(\mathbb{T}^n)$
\begin{align}\label{weis}
 \frac{1}{\sqrt{2}^m}
\Big(\int_{\mathbb{T}^n} |P(z)|^2\,dz\Big)^\frac{1}{2}
\leq\int_{\mathbb{T}^n} |P(z)|\,dz\,.
\end{align}

In other terms, the set $\big\{ z^\alpha \colon \alpha \in \Lambda_1(\leq m,n)\big\}$ of characters in $\widehat{\mathbb{T}^n}
= \mathbb{Z}^n $ forms a $\Lambda(2)$-set with constant $C_2 \leq \sqrt{2^m}$ (recall from \eqref{index2} the definition of
$\Lambda_1(\leq m,n)$, which clearly  may be identified  with the set of characters $ z^\alpha$ it generates).

Then the following result is an immediate consequence of  Theorem~\ref{C(G)proj}, Corollary~\ref{corB2}, and~\eqref{weis}.

\begin{corollary}  \label{manydimensionsB}
Let $J \subset \mathbb{Z}_0^n$ be a finite set. Then
\[
\boldsymbol{\lambda}\big(\text{Trig}_{J}(\mathbb{T}^n)\big) = \int_{\mathbb{T}^n} \Big|\sum_{\alpha\in J } z^\alpha\Big|\,dz\,,
\]
and if  $\Lambda \subset \Lambda_1(\leq m,n)$, then
\[
\frac{1}{\sqrt{2^m}}\sqrt{|\Lambda |}  \leq \boldsymbol{\lambda}\big(\text{Trig}_{\Lambda }(\mathbb{T}^n)\big) \leq  \sqrt{|\Lambda |}\,.
\]
\end{corollary}

\noindent In the context of this corollary the following  formula and estimate
\begin{equation} \label{cardi}
|\Lambda_1(m,n)| = \dbinom{n+m-1}{m}\leq e^m \Big( 1 + \frac{n}{m}\Big)^m
\end{equation}
for the cardinality of $\Lambda_1(m,n)$ is very useful.

\subsubsection{Products of groups} \label{The Lozinski and Kharishiladze revisited}

The aim is to prove various multivariate  variants of a famous result by Lozinski-Kharshiladze
(see \cite[IIIB. Theorem 22]{wojtaszczyk1996banach} and \cite{natanson1961constructive}),
which shows a  precise estimate of the projection constant of $\text{Trig}_{\leq m}(\mathbb{T})$, the space of trigonometric polynomials on
$\mathbb{T}$ of degree at most $m$ equipped with the sup-norm:
\begin{equation} \label{LoKa}
\boldsymbol{\lambda}\big(\text{Trig}_{\leq m}(\mathbb{T})\big)= \frac{4}{\pi^2} \log(m+1) + o(1).
\end{equation}
 This is done by mostly routine  applications of \,Theorem \ref{C(G)proj}. We point out
that these results in Section~\ref{Projection constants II} are the key  to obtain formulas and estimates for the
projection constants of spaces of Dirichlet polynomials.

\begin{proposition}
\label{product}
Let $G:=G_1 \times \cdots \times G_n$  be  a product of compact abelian groups $G_1, \ldots, G_n$ $($with the Haar measures
$\mathrm{m}_j, 1\leq j\leq n )$, and   $E:= \big\{\gamma^{\alpha} \colon \alpha \in J_\infty(\leq m,n)\big\} \subset \widehat{G}$, where
$(\gamma_1, \ldots, \gamma_n) \in \widehat{G_1} \times \cdots \times \widehat{G_n}$
and the definition of  $J_\infty(\leq m,n)$ is given in~\eqref{index2}.
Then
\[
\boldsymbol{\lambda}\big(\text{Trig}_{{E}}(G)\big) =
\prod_{j=1}^n \int_{G_j} \Big|\sum_{k \in \mathbb{Z}:|k|\leq m} \gamma_j(x_j)^k\Big|\,d{\mathrm{m}}_j(x_j)\,.
\]
\end{proposition}

\begin{proof}
Let $\mu$ be the Haar measure on $G$, which is nothing else then the product of the Haar measures $\mathrm{m}_j$.
Then Theorem \ref{C(G)proj} yields
\[
\boldsymbol{\lambda}\big(\text{Trig}_{{E}}(G)\big) = \int_G \Big|\sum_{\alpha\in J_\infty(\leq m,n)} \gamma^{\alpha}(x) \Big|\,d\mu(x)\,.
\]
Clearly, for all $x=(x_1,\ldots, x_n) \in G_1 \times \cdots \times G_n$ one has
\[
\sum_{\alpha \in J_\infty(\leq m,n)} \gamma^\alpha(x)  = \prod_{j=1}^n \sum_{|k|\leq m} \gamma_j(x_j)^k\,,
\]
and hence
\[
\boldsymbol{\lambda}\big(\text{Trig}_{{E}}(G)\big) = \int_G \prod_{j=1}^n \Big|\sum_{|k|\leq m} \gamma_j(x_j)^k\Big|\,d\mu(x)\,.
\]
Fubini's theorem finishes the argument.
\end{proof}

We also use  the following more general variant.

\begin{proposition}
\label{product corollary}
Let $G:=G_1 \times \cdots \times G_n$  be  a product of compact abelian groups $G_1, \ldots, G_n$. Moreover, let 
 $I:=I_1\times \dots\times I_n\subset \mathbb{Z}^n$, with $I_j$ finite subsets of \, $\mathbb Z$ for $j\in \{1,\ldots,n\}$, and
$E:= \{\gamma^{\alpha} \colon \alpha \in I\}$. Then
\[
\boldsymbol{\lambda}\big(\text{Trig}_{{E}}(G)\big) = \prod_{j=1}^n \int_{G_j} \Big|\sum_{k\in I_j}
\gamma_j(x_j)^k\Big|\,d{\mathrm{m}}_j(x_j)\,.
\]
\end{proposition}

As announced above, we finally present multivariate variants of the the Lozinski-Kharshiladze result mentioned in \eqref{LoKa}.
For each $m\in \mathbb{N}_0$,  let $D_m$  be the $m$th Dirichlet kernel 
$$D_m(t):= \sum_{k=-m}^{m} e^{-ikt}\,,$$ 
and $L_m$ the $m$th
Lebesgue constant given by
\[
L_m:= \frac{1}{2\pi} \int_0^{2\pi} |D_m(t)|\,dt
= \frac{1}{2\pi} \int_0^{2\pi} \bigg|\frac{\sin (m + \frac{1}{2})t}{\sin\frac{t}{2}}\bigg|\,dt\,.
\]
We recall the well-known standard estimates
\begin{equation}\label{Lebconst}
\frac{4}{\pi^2} \log(m +1) \,\,< \,\,L_m \,\,<\,\,
3 + \log m, \quad\, m\in \mathbb{N}\,.
\end{equation}

\begin{corollary}  \label{Lozinski-Kharshiladze}
For $\mathbf{d}=(d_1,\dots,d_n)\in \mathbb{N}_0^n$ let $I_\mathbf{d}:= \big\{\alpha
\in \mathbb{Z}^n \colon \, |\alpha_j| \leq d_j, \,\,1\leq j \leq n\big\}$. Then
\[
\boldsymbol{\lambda}\big(\text{Trig}_{I_{\mathbf{d}}}(\mathbb{T}^n)\big) = \prod_{j=1}^n L_{d_j}\,.
\]
\end{corollary}

\begin{proof}
Applying Proposition~\ref{product corollary}, we get
\begin{align*}
\boldsymbol{\lambda}\big(\text{Trig}_{I_{\mathbf{d}}}(\mathbb{T}^n)\big) = \prod_{j=1}^n \int_{\mathbb{T}} \Big|\sum_{|\alpha_j|\le d_j } z^{\alpha_j}\Big|\,dz
= \prod_{j=1}^n \int_{\mathbb{T}} |D_{d_j}(z_j)|\,dz_j = \prod_{j=1}^n L_{d_j}\,,
\end{align*}
as required.
\end{proof}

In order to state the 'analytic counterpart' of Corollary~\ref{Lozinski-Kharshiladze}, we define for each $m\in \mathbb{N}_0$,
\[
\text{$D_m^{+}(t):= \sum_{k=0}^{m} e^{-ikt}$ \,\,\,\,and \,\,\,\, $L_m^{+}:= \frac{1}{2\pi} \int_0^{2\pi} |D_m^{+}(t)|\,dt$\,.}
\]
and recall from \eqref{index2} the notation $\Lambda_\infty(\leq m,n):= \big\{\alpha \in \mathbb{N}_0^n : \,\, \max_j \alpha_j \leq m \big\}$.

\begin{corollary} \label{LoKha}
For $m,n\in \mathbb{N}$
\begin{itemize}
\item[\rm{(i)}] $\boldsymbol{\lambda}\big(\text{Trig}_{\Lambda_\infty(\leq m,n)}(\mathbb{T}^n)\big) = (L_m^{+})^n$\,,
\vspace{2mm}
\item[\rm{(ii)}] $\boldsymbol{\lambda}\big(\text{Trig}_{\Lambda_\infty(\leq 2m,n)}(\mathbb{T}^n)\big) = (L_m)^n$\,,
\vspace{2mm}
\item[\rm{(iii)}] $(L_m - 1)^n \leq \boldsymbol{\lambda}\big(\text{Trig}_{\Lambda_\infty(\leq 2m+1,n)}(\mathbb{T}^n)\big) \leq (L_m + 1)^n$\,,
    \vspace{2mm}
\item[\rm{(iv)}] $\boldsymbol{\lambda}\big(\text{Trig}_{\Lambda_\infty(\leq m,n)}(\mathbb{T}^n)\big) \sim (1 + \log m)^n$\,.
\end{itemize}
\end{corollary}

\begin{proof}
\noindent (i) By Proposition \ref{product}, we have as desired
\begin{align*}
\boldsymbol{\lambda}\big(\text{Trig}_{\Lambda_\infty(\leq m,n)}(\mathbb{T}^n)\big)
=  \prod_{j=1}^n \int_{\mathbb{T}} \Big|\sum_{0\le \alpha_j \le m} z^{\alpha_j}\Big|\,dz
= \prod_{j=1}^n \int_{\mathbb{T}} |D_m^{+}(z_j)|\,dz_j = (L_m^{+})^n\,.
\end{align*}
(ii) It is easy to check that
\[
D_{2m}^{+}(t) = e^{-imt}\,\frac{\sin (m + \frac{1}{2})t}{\sin\frac{t}{2}}, \quad\, t \in (0, 2\pi)\,.
\]
This implies that $L_m = \|D_m\|_{L_1(\mathbb{T})} = \|D_{2m}^{+}\|_{L_1(\mathbb{T})}=L_{2m}^+$, so the statement follows
from (i).

\noindent (iii) The statement follows by (ii) combined with the obvious estimates
\[
\|D_{2m}^{+}\|_{L_1(\mathbb{T})} - 1 \leq \|D_{2m+1}^{+}\|_{L_1(\mathbb{T})} \leq \|D_{2m}^{+}\|_{L_1(\mathbb{T})} + 1,
\quad\, m\in \mathbb{N}\,.
\]
(iv) Combining the estimates from \eqref{Lebconst} with those from  (ii) and (iii), we get the required
equivalence.
\end{proof}

We conclude with the following two limit formulas; see again \eqref{index2} for the definition of the index sets
$J_\infty(\leq m,n)$ and $\Lambda_\infty(\leq m,n)$.

\begin{corollary} \label{limit formula}
For each $m,n\in \mathbb{N}$
\[
\lim_{m\to \infty} \frac{\boldsymbol{\lambda}\big(\text{Trig}_{J_\infty(\leq m,n)}(\mathbb{T}^n)\big)}{\log^n m}
=
\lim_{m\to \infty} \frac{\boldsymbol{\lambda}\big(
\text{Trig}_{\Lambda_\infty(\leq m,n)}(\mathbb{T}^n)
\big)}{\log^n m} = \Big(\frac{4}{\pi^2}\Big)^n\,.
\]
\end{corollary}

\begin{proof}
Recall the well-known formula
\begin{equation}\label{lebformula}
\lim_{m \to \infty} \frac{L_m}{\log m} = \frac{4}{\pi^2}\,.
\end{equation}
Then the first limit follows from Corollary~\ref{Lozinski-Kharshiladze}.
For the proof of the second limit we distinguish the two sequences
$\big(\boldsymbol{\lambda}\big(
\text{Trig}_{\Lambda_\infty(\leq 2m,n)}(\mathbb{T}^n)
\big)/ \log (2m)\big)_m$
and $\big(\boldsymbol{\lambda}\big(
\text{Trig}_{\Lambda_\infty(\leq 2m+1,n)}(\mathbb{T}^n)
\big)/ \log (2m+1)\big)_m$.
Observe that both sequences converge to $\frac{4}{\pi^2}$. Indeed, by  Corollary~\ref{LoKha} (ii) we have that
\[
\boldsymbol{\lambda}\big(
\text{Trig}_{\Lambda_\infty(\leq 2m,n)}(\mathbb{T}^n)
\big) = (L_{m})^n\,.
\]
So, using \eqref{lebformula}, gives the claim for the first sequence. The second sequence  is handled the
same way using Corollary~\ref{LoKha} (iii).
\end{proof}

\subsubsection{Lebesgue constants}

We point out that although we in Corollary~\ref{manydimensionsB} prove an integral formula for  the projection constant
of trigonometric polynomials on $\mathbb{T}^n$ (supported on the index set $J$), we usually will not be able to compute
that integral explicitly.

In fact the integral from Corollary~\ref{manydimensionsB} appears naturally  in multivariate Fourier analysis (due to
different summation methods of multivariate Fourier series), and in this context it is usually  called  Lebesgue constant
(of the index set $J$).

Let us illustrate this with an example of particular interest. Given $m,n \in \mathbb{N}$, the index set
\begin{equation} \label{s-set}
J_2(\leq m,n) = \Big\{\alpha \in \mathbb{Z}^n:\, \sum_{j=1}^n |\alpha_j|^2 \leq m^2\Big\}\,,
\end{equation}
of all  spherical multi indices $\alpha \in \mathbb{Z}^n$ of 'Euclidean order' less or equal  $m$
(see again~\eqref{index2}) plays an outstanding  role when multivariate Fourier series are summed up by  their spherical
partial sums. Obviously, we have that
\[
\boldsymbol{\lambda}\big(\text{Trig}_{J_2(\leq m,n)}(\mathbb{T}^n)\big)
\leq \sqrt{\text{dim}\big(\text{Trig}_{J_2(\leq m,n)}(\mathbb{T}^n)\big)} = \sqrt{|J_2(\leq m,n)|}\,;
\]
this can either be seen as an immediate  consequence of Corollary~\ref{manydimensionsB}
or the Kadec-Snobar theorem.

However, to find a precise  formula for the dimension of $\text{Trig}_{J_2(\leq m,n)}(\mathbb{T}^n)$, or equivalently
the cardinality of $J_2(\leq m,n)$, is in general a~highly non-trivial problem.

The so-called  $n$-dimensional sphere problem of classical lattice point theory concerns  the number $\mathcal{N}_n(R)$
of lattice points in the closed Euclidean ball of radius $R$ in $\mathbb{R}^n$. Let us mention here that in the case of
$n>3$ one has
\[
\mathcal{N}_n(R) = \omega_n R^{n} + O\big(R^{n-2}\big)\,,
\]
where $\omega_n := \frac{2\pi^{n/2}}{\Gamma(n/2)}$ is the volume of the unit ball in $\mathbb{R}^n$. For $n =3$
Heath-Brown proves in \cite{Brown} that
\[
\mathcal{N}_3(R) = \frac{4}{3}\pi R^3 + O\big(R^{\frac{21}{16} + \varepsilon}\big)\,,
\]
where $\varepsilon >0$ is arbitrary. His  proof is based on an explicit formula for  the number of ways to write
$n\in \mathbb{N}$ as a~sum of three squares of integers. To get this, a~Dirichlet
$L$-series $L(1,\chi_n)=\sum_{m=1}^{\infty}\chi_n m^{-s}$ with $\chi_n(m)= \big(-\frac{4n}{m}\big)$ is studied.
We also  refer to the remarkable article \cite{Gotze} by G\"otze, where the  lattice point problem is studied for
the more complicated case of $n$-dimensional ellipsoids.

Under the above notation of~\eqref{s-set} we are ready to state the following corollary.

\begin{corollary}
There exist positive constants $C_1$ and $C_2$ with $C_1 <C_2$ depending on $n$ such that
\[
C_1\,m^{\frac{n-1}{2}} \leq \,  \boldsymbol{\lambda}\big(\text{Trig}_{J_2(\leq m,n)}(\mathbb{T}^n)\big) \leq \,  C_2\,m^{\frac{n-1}{2}}\,.
\]
\end{corollary}

\begin{proof} We use a two-sided estimate proved by Babenko \cite{Babenko}, which (in our notation) states
there  exist constants $C_1$ and $C_2$ with $0<C_1 <C_2$ depending on $n$, such that the following estimates hold for each $m \in \mathbb{N}$,
\[
C_1\,m^{\frac{n-1}{2}} \leq \, \int_{\mathbb{T}^n} \Big|\sum_{\alpha\in J_2(\leq m,n)} z^\alpha \Big|\,dz \,
\leq \, C_2 \, m^{\frac{n-1}{2}}\,.
\]
Thus the required statement follows from
Corollary~\ref{manydimensionsB}.
\end{proof}

For the history of Babenko's important result, we refer to the interesting survey paper by Liflyand \cite[Theorem 1.1]{Liflyand}.
We note that the proof given there, does not give precise asymptotic estimates in the dimension  $n$ for the  constants $C_1$ and
$C_2$. The estimates of both $C_1$ and $C_2$ depend deeply on the Fourier transform of the indicator function of the closed
Euclidean ball of $\mathbb{R}^n$, which is expressed by the well-known Bessel function. It seems interesting to note that
$C_1= C_1(\varepsilon):= C_4 \varepsilon^{(n-1)/2} - C_3 \varepsilon^{n/2}$ for small $\varepsilon \in (0, 1)$.

\bigskip

\section{Dirichlet polynomials} \label{Integral formula-Dirichlet polynomials}

Recall from the introduction that, given a frequency $\omega = (\omega_n)$ and a finite index set $J \subset \mathbb{N}$, we write $\mathcal{H}_\infty^{J}(\omega)$ for the  Banach space of all $\omega$-Dirichlet polynomials supported on $J$  endowed with the
supremum norm on the imaginary line. The following integral formula for the projection constant of $\mathcal{H}_\infty^{J}(\omega)$
is one of our main contributions, and crucial for all further descriptions of $\boldsymbol{\lambda}\big(\mathcal{H}_\infty^{J}(\omega)\big)$
for concrete frequencies and index sets. See also Theorem~\ref{main-dirB} for an equivalent formulation in terms of compact abelian groups.

\begin{theorem} \label{main-dir}
Let  $J \subset \mathbb{N}$ be a finite subset and   $\omega= (\omega_n)$  a~frequency. Then
\begin{align*}
\boldsymbol{\lambda}\big(\mathcal{H}_\infty^{J}(\omega)\big) =
\lim_{T \to \infty} \frac{1}{2T} \int_{-T}^T \Big|\sum_{n \in J} e^{-i\omega_n t}\Big|\,dt\,.
\end{align*}
\end{theorem}

The proof, given at the end of the following section, is mainly based on Theorem~\ref{C(G)proj} and a reformulation of $\mathcal{H}_\infty^{J}(\omega)$ in terms of a Hardy type space of trigonometric polynomials on certain compact abelian groups associated to $\omega$,
which are all supported on a finite set of characters associated to $J$.

\subsection{Bohr's vision} \label{Bohr's vision}
We need to adapt ideas from  \cite{defantschoolmann2019Hptheory}. Our first aim is to
describe the finite dimensional Banach space $\mathcal{H}_\infty^{J}(\omega)$ in terms of a~Hardy type space of functions on
a~certain compact abelian group.

In what follows, a~pair $(G, \beta)$ is said to be a~Dirichlet group if $G$ is a~compact abelian group and
$\beta \colon \mathbb{R} \to G$  a~continuous homomorphism with dense range; as before we  denote the Haar measure
on $G$ by $\mathrm{m}$. Then, by density,  the  dual map
\[
\widehat{\beta}\colon \widehat{G} \to \widehat{\mathbb{R}}
\]
is injective. Moreover, if we write  $\mathbb{R}$ for the group $(\mathbb{R},+)$ endowed with its natural topology, then the
homomorphism $ \mathbb{R}  = \widehat{\mathbb{R}} \,, \,\,\,x \mapsto e^{ix\,\bullet}$ identifies  $\mathbb{R}$ and $\widehat{\mathbb{R}}$
as topological groups, and hence $\widehat{\beta}(\widehat{G} )$ may be interpreted as a  subset of $\mathbb{R}$.

We in particular observe  that if $(G, \beta)$ is a~Dirichlet group and $\gamma \colon G\to \mathbb{T}$  a~character, then the composition
$\gamma \circ \beta \colon \mathbb{R} \to \mathbb{T}$ is a~character on  $\mathbb{R}$. Thus there exists a~unique $x\in \mathbb{R}$ such
that $\gamma \circ \beta(t) =e^{ixt}$ for all $t\in \mathbb{R}$. The following  result from \cite[Proposition~3.10]{defantschoolmann2019Hptheory}
is crucial - we include a~proof for the sake of completeness.

\begin{proposition} \label{basic}
For any Dirichlet group $(G, \beta)$ one has
\[
\int_G f(a)\,d\mathrm{m}(a) = \lim_{T \to \infty}\, \frac{1}{2T} \int_{-T}^T f\circ \beta(t)\,dt, \quad\, f\in C(G)\,.
\]
\end{proposition}

\begin{proof} Note first that for all  $\gamma\in \widehat{G}$ we have
\[
\int_G \gamma(a)\,d\mathrm{m}(a) = \lim_{T \to \infty}\, \frac{1}{2T} \int_{-T}^T \gamma\circ \beta(t)\,dt
\,.
\]
Indeed, as explained above  by the injectivity of the dual map $\widehat{\beta}$, we know that for every  $\gamma \in \widehat{G}$
there is a unique $x = x(\gamma) \in \mathbb{R}$ such that $\gamma \circ \beta(t) = e^{ixt}$ for all $t\in \mathbb{R}$. Now recall
that $\int_G \gamma\,d\mathrm{m}= 0$ \big(resp., $\int_G \gamma\,d\mathrm{m}= 1$\big) whenever $\gamma \neq 1$ (resp., $\gamma =1$). Clearly,
$\gamma \neq 1$ (resp., $\gamma = 1$) yields $x \neq 0$ (resp., $x =0$), so the above equality is obvious in this case. As a~consequence,
the claim also holds for all trigonometric polynomials  $f \in C(G)$. But since all trigonometric polynomials are dense in
$C(G)$, and moreover the set $\{\Phi_T \colon T >0\}$ of all linear functionals
\[
\Phi_T\colon  C(G) \to \mathbb{C}\,, \quad\, \Phi_T(f):= \frac{1}{2T} \int_{-T}^T f \circ \beta(t)dt\,
\]
is uniformly bounded, the conclusion is a~consequence of the Banach-Steinhaus theorem.
\end{proof}

Given a~frequency $\omega=(\omega_n)$,  we need Dirichlet groups which are adapted to $\omega$. We call the pair $(G,\beta)$
a~$\omega$-Dirichlet group, whenever $\{\omega_n \colon n \in \mathbb{N}\}\subset \widehat{\beta}(\widehat{G}) $, that is, for every
character $e^{-i\omega_n \bullet}: \mathbb{R}\to \mathbb{T}$ there is a~character $h_{\omega_n} \in \widehat{G}$ (which then is unique)
such  that the following diagram commutes:
\begin{equation*}
\begin{tikzpicture}[scale = 0.8]
        \node (G) at (0,0) {$G$};
        \node (T) at (3,0) {$\mathbb{T}$};
        \node (R) at (0,-2) {$\mathbb{R}$};

        \draw[-latex] (G) -- node[above] {$h_{\omega_n}$} (T);
        \draw[-latex] (R) -- node[below right] {$e^{-i\omega_n \bullet}$} (T);
        \draw[-latex] (R) -- node[left] {$\beta$} (G);
    \end{tikzpicture}
    \end{equation*}
For all needed information on $\omega$-Dirichlet groups see again \cite{defantschoolmann2019Hptheory}.

Let us collect some examples.
The very first is the  frequency $\omega = (n)$ for which the pair $(\mathbb{T},\beta_{\mathbb{T}} )$ with
\[
\beta_{\mathbb{T}} : \mathbb{R} \rightarrow \mathbb{T} \,, \,\, t \mapsto e^{-it}
\]
obviously forms a $\omega$-Dirichlet group. Identifying  $\widehat{\mathbb{T}} = \mathbb{Z}$ we get
that $h_n(z) = z^n$ for $z \in \mathbb{T}, n \in \mathbb{Z}$.

The second example  shows that $\omega$-Dirichlet groups for any possible frequency $\omega$ always exist.
In fact, given a frequency $\omega$, the so-called Bohr compactification of the reals  forms a $\omega$-Dirichlet group.

\begin{example}
\label{examples1}
Denote by $\overline{\mathbb{R}}:=\widehat{(\mathbb{R},+,d)}$  the   Bohr compactification of $\mathbb{R}$, where $d$ stands
for the discrete topology. This is a compact abelian group,  which together with the mapping
\begin{equation*}
\beta_{\overline{\mathbb{R}}}\colon \mathbb{R} \hookrightarrow \overline{\mathbb{R}},\,\, x \mapsto \left[ t \mapsto e^{-ixt}\right]
\end{equation*}
forms a $\omega$-Dirichlet group for every frequency $\omega$.
\end{example}

But for  concrete  frequencies $\omega$, there often are $\omega$-Dirichlet groups which in a~sense reflect their structure more
naturally than the Bohr compactification.

\begin{example} \label{examples2}
Let $\omega = (\log n)$. Then the infinite dimensional torus
$\mathbb{T}^\infty$ together with the so-called Kronecker flow
\[
\beta_{\mathbb{T}^\infty}\colon  \mathbb{R}  \rightarrow \mathbb{T}^\infty\,, \, \quad t \mapsto (\mathfrak{p}_k^{-it})\,,
\]
where $\mathfrak{p}_k$ again denotes the $k$th prime, forms a $\omega$-Dirichlet group.  This is basically a~consequence of  Kronecker's approximation
theorem (for an alternative 'harmonic analysis' argument see \cite[Example 3.7]{defantschoolmann2019Hptheory}). Note that, identifying
$$\mathbb{Z}^{(\mathbb{N})}
=
\widehat{\mathbb{T}^\infty}\,,
\quad \alpha \mapsto [z \mapsto z^\alpha] $$
($\mathbb{Z}^{(\mathbb{N})}$ all finite sequences of integers),  for every $z \in \mathbb{T}^\infty$ one has
\begin{equation}\label{bohrtrafo}
h_{ \log n}(z)=h_{\sum \alpha_j \log \mathfrak{p}_j}(z) = z^\alpha\,,
\end{equation}
where $n=\mathfrak{p}^\alpha$ with $\alpha \in \mathbb{Z}^{(\mathbb{N})}$.
\end{example}

As announced, we now reformulate the Banach space $\mathcal{H}_\infty^{J}(\omega)$  of all $\omega$-Dirichlet polynomials
supported on the finite index set $J \subset \mathbb{N}$ as a Hardy space of functions on a $\omega$-Dirichlet group.
Given a~finite set $J \subset \mathbb{N}$ and a~$\omega$-Dirichlet group $(G,\beta)$, we
write
\[
H_\infty^{\omega, J}(G)
\]
 for the subspace of all trigonometric  polynomials
$f= \sum_{n \in J} \hat{f}(h_{\omega_n}) h_{\omega_n}$ in $L_\infty(G)$ (with respect to the Haar measure $\mathrm{m}$ on $G$).

The following equivalent description of $\mathcal{H}_\infty^{J}(\omega)$
is now obvious.

\begin{proposition} \label{H=H}
Let  $J \subset \mathbb{N}$ be a~finite subset, and $(G,\beta)$ a~$\omega$-Dirichlet group. Then the Bohr map
\[
\mathcal{B}: \mathcal{H}_\infty^{J}(\omega) \rightarrow  H_\infty^{\omega, J}(G), \quad\,
\sum_{n \in J} a_n e^{-\omega_n s} \mapsto \sum_{n \in J} a_n h_{\omega_n}
\]
defines an isometric linear bijection which preserves Bohr and Fourier coefficients.
\end{proposition}

\begin{proof}
The collections $(e^{-i\omega_n \bullet})_{n\in J}$ and $(h_{\omega_n})_{n\in J}$ are both linearly independent. Hence it  suffices to check that for every collection $(a_n)_{n\in J}$ of complex scalars, we have
\[
\sup_{t \in \mathbb{R}} \Big|\sum_{n \in J} a_n e^{-i\omega_n t}\Big|
= \sup_{g \in G} \Big|\sum_{n \in J} a_n h_{\omega_n}(g)\Big|\,.
\]
 But this is clear by the fact that  $e^{-i\omega_n \bullet} =
h_{\omega_n} \circ \beta$ for each $n \in J$, and $\beta$ moreover has dense range.
\end{proof}

Finally, we are in the position to prove Theorem~\ref{main-dir}.

\begin{proof}[Proof of Theorem~\ref{main-dir}]
Choose some $\omega$-Dirichlet group $(G, \beta)$ (this is possible by  Example~\ref{examples1}). Then by
Proposition~\ref{H=H} and Theorem~\ref{C(G)proj} one has
\[
\boldsymbol{\lambda}\big(\mathcal{H}_\infty^{J}(\omega)\big) = \boldsymbol{\lambda}\big( H_\infty^{\omega, J}(G)\big)
= \int_G \Big|  \sum_{n \in J} h_{\omega_n}\Big|\,d\mathrm{m}\,.
\]
Applying now Proposition~\ref{basic}, we get
\[
\int_G \Big|  \sum_{n \in J} h_{\omega_n}\Big|\,d\mathrm{m}
= \lim_{T \to \infty} \frac{1}{2T} \int_{-T}^T \Big|\sum_{n \in J} e^{-i\omega_n t}\Big|\,dt\,,
\]
 so this completes the proof.
\end{proof}

We add a 'group version' of Theorem~\ref{main-dir}, which in view of Proposition~\ref{basic} is immediate.

\begin{theorem} \label{main-dirB}
Let  $J \subset \mathbb{N}$ be a finite subset, $\omega$ a~frequency, and $(G,\beta)$ a~$\omega$-Dirichlet group
with Haar measure $\mathrm{m}$. Then
\begin{align*}
\boldsymbol{\lambda}\big(\mathcal{H}_\infty^{J}(\omega)\big) =  \int_G \Big|  \sum_{n \in J} h_{\omega_n}\Big|\,d\mathrm{m}\,.
\end{align*}
\end{theorem}

\subsection{Projection constants - Dirichlet polynomials} \label{Projection constants II}
Our aim now is to use the knowledge collected so far, to produce  concrete applications.
In fact, we distinguish three cases -- each motivated slightly differently.

\subsubsection{{\bf Case I}}
We consider Banach spaces $\mathcal{H}_\infty^{\leq x}(\omega)$ of Dirichlet polynomials of length $x$ for the three standard frequencies
$\omega = (n)_{n \in \mathbb{N}_0}$, $\omega = (\log p_n)_{n \in \mathbb{N}}$ (where as above $p_n$ is the $n$th prime number), and $\omega = (\log n)_{n \in \mathbb{N}}$.

Starting with $\omega = (n)_{n \in \mathbb{N}_0}$, for $x \in \mathbb{N}$ the identification
\begin{equation} \label{idy}
  \mathcal{H}_\infty^{\leq x}(\omega) \,\, = \,\, \text{Trig}_{\{n\in \mathbb{N}_0\colon n \leq x\}}(\mathbb{T})
\,, \,\,\,\,\,\, \sum_{n=0}^x a_{n} e^{-\omega_n s} \mapsto  \sum_{n=0}^{x} a_{n} z^{n}
\end{equation}
is obviously isometric (doing the variable transformation $z = e^{-s}$), and hence by Theorem \ref{main-dirB} we get  the  integral formula
\begin{align} \label{inte1}
\boldsymbol{\lambda}\big(\mathcal{H}_\infty^{\leq x}(\omega)\big) =
\int_{\mathbb{T}} \Big|\sum_{n=0}^{x}  z^k\Big|\,dz\,.
\end{align}
Then a standard calculation  leads to the following result, which in view of \eqref{idy} is a counterpart of the  Lozinski-Kharshiladze theorem from \eqref{LoKa}.

\begin{theorem}\label{inte2} For the frequency $\omega = (n)_{n \in \mathbb{N}_0}$
  \begin{align*}
\boldsymbol{\lambda}\big(\mathcal{H}_\infty^{\leq x}(\omega)\big)=  \frac{4}{\pi^2} \log(x+1) + o(1)\,.
\end{align*}
 \end{theorem}

Next we consider the  frequency $\omega = (\log p_n)$. This is the  canonical example of a $\mathbb{Q}$-linearly
independent frequency. It is well-known (see, e.g., \cite[Theorem 9.2]{katznelson2004introduction})
that, if
$\omega_1, \ldots, \omega_n$ are real numbers,  which are linearly independent over $\mathbb{Q}$, and $\omega_0=0$,
then for any choice of complex numbers $a_0, a_1, \ldots , a_x$ we have
\begin{equation*}\label{bohr-in}
\sum_{n=0}^x |a_n|  =  \sup_{ t \in \mathbb{R}} \Big|\sum_{n=0}^x a_n e^{-i \omega_n t}\Big|\,.
\end{equation*}
In other terms, for any frequency $\omega$, which is linearly independent over $\mathbb{Q}$, and for any finite subset
$J \subset \mathbb{N}$ the identification
\[
\mathcal{H}_\infty^{J}(\omega) \,\, = \,\, \ell_1^{|J|}\,, \,\,\quad \sum_{n \in J} a_n e^{-\omega_n s} \mapsto (a_n)_{n \in J}
\]
is isometric. Then the following integral and limit formula is an  immediate consequences of \eqref{grunbuschC-B}
and~\eqref{koenigschuetttomczak}.

\begin{theorem} \label{logp}
  Let $\omega $ be a $\mathbb{Q}$-linearly independent frequency and  $J \subset \mathbb{N}$ a finite subset. Then
\begin{align*}
\boldsymbol{\lambda}\big(\mathcal{H}_\infty^{J}(\omega)\big) =    \int_{\mathbb{T}^{|J|}} \Big|\sum_{n \in J}  z_k\Big| dz\,.
\end{align*}
Moreover,
\begin{equation*}
\lim_{x \to \infty}  \frac{\boldsymbol{\lambda}\big(\mathcal{H}_\infty^{\leq x}(\omega)\big)}{\sqrt{x}}
= \frac{\sqrt{\pi}}{2}\,.
\end{equation*}
\end{theorem}

We may, despite the precise constant,  extend the preceding result within the setting of $\Lambda(2)$-sets explained
in~Section~\ref{L2}.
Recall from Remark~\ref{B2} that every $B_2$-set of characters on a compact abelian group  is a $\Lambda(2)$-set
with $\Lambda(2)$-constant $\leq \sqrt{2}$.

Hence, Theorem~\ref{main-dirB}, Corollary~\ref{corB2}, and Remark~\ref{B2} imply the following
extension of Theorem~\ref{logp}.

\begin{theorem}\label{b2dirichelet}
Let $\omega$ be a frequency and $(G,\beta)$ a~$\omega$-Dirichlet group with the property that all characters
$h_{\omega_n} \in \widehat{G}$ form a  $\Lambda(2)$-set with constant $C_2(\omega)$.  Then
\begin{equation*}
\frac{1}{C_2(\omega)}\sqrt{x} \,\, \leq \,\,
\boldsymbol{\lambda}\big(\mathcal{H}_\infty^{\leq x}(\omega)\big)   \,\, \leq \,\,  \sqrt{x} \,.
\end{equation*}
In particular, if $\omega$  satisfies the $B_2$-condition, then
\begin{equation*}
\frac{1}{\sqrt{2}} \sqrt{x}
\,\, \leq \,\,
\boldsymbol{\lambda}\big(\mathcal{H}_\infty^{\leq x}(\omega)\big) \,\, \leq \,\,  \sqrt{x} \,.
\end{equation*}
\end{theorem}

Finally, we handle the ordinary frequency $\omega = (\log n)$ - getting one of our main applications for free.
Indeed, combining Harper's deep result from \eqref{harperA} with Theorem~\ref{main-dir}, we obtain the precise asymptotic order of
the projection constant of the  Banach space $\mathcal{H}_\infty^{\leq x}$.

\begin{theorem}\label{harpo}
\begin{align*}
\boldsymbol{\lambda}\big(\mathcal{H}_\infty^{\leq x}\big( (\log n)\big)\big)   =
O\left(\frac{\sqrt{x}}{(\log \log x)^{\frac{1}{4}}}\right)
\end{align*}
\end{theorem}

It seems interesting to rephrase this result again in terms of trigonometric polynomials.
We write
\[
\Delta(x) = \big\{\alpha \in \mathbb{N}_0^{\pi(x)}\colon 1 \leq \mathfrak{p}^\alpha \leq x\big\}\,,
\]
where $x \geq 1$ and $\pi(x)$  as usual counts the primes $\mathfrak{p} \leq x$. Then by Proposition~\ref{H=H} we have that
\begin{equation}\label{deltax}
  \mathcal{H}_\infty^{\leq x}\big( (\log n)\big) = \text{Trig}_{\Delta(x)}\big(\mathbb{T}^{\pi(x)}\big)\,,
\end{equation}
where the identification is given by the Bohr transform -- being isometric and coefficient preserving. Hence the following two results
are immediate consequences of Theorem~\ref{harpo} (first statement) and Theorem~\ref{C(G)proj}
(second statement).
\begin{corollary} \label{harpo2}
\[
\boldsymbol{\lambda}\big(\text{Trig}_{\Delta(x)}\big(\mathbb{T}^{\pi(x)}\big)\big)   =
O\left(\frac{\sqrt{x}}{(\log \log x)^{\frac{1}{4}}}\right)\,.
\]
Moreover,
\begin{align*}
\boldsymbol{\lambda}\big(\text{Trig}_{\Delta(x)}\big(\mathbb{T}^{\pi(x)}\big)\big)
\,= \,\big\|\mathbf{P}_{\Delta(x)}: C(\mathbb{T}^{\pi(x)}) \rightarrow \text{Trig}_{\Delta(x)}\big(\mathbb{T}^{\pi(x)}\big)\big\|
\,= \,\int_{\mathbb{T}^{\pi(x)}} \big| \sum_{\alpha \in \Delta(x)} z^\alpha \big| d z\,,
\end{align*}
where $\mathbf{P}_{\Delta(x)}$ stands for the restriction of the orthogonal projection on $L_2(\mathbb{T}^{\pi(x)})$
onto $\text{Trig}_{\Delta(x)}\big(\mathbb{T}^{\pi(x)}\big)$\,.
\end{corollary}

\subsubsection{{\bf Case II}} \label{Case II}

For the ordinary frequency  $\omega = (\log n)$  we look at Banach spaces $\mathcal{H}_\infty^J(\omega)$ of  Dirichlet polynomials
generated by index sets $J$ of natural numbers $k$, which have an a~priori specified  complexity of their prime number decompositions.
More precisely, the first two results consider index sets $J$, where each element $k=\mathfrak{p}^\alpha \in J$ incorporates not more
than $n$ different prime divisors with   $\max \alpha_j \leq m$, or similarly, $n$ different  prime divisors with
$|\alpha| = \sum \alpha_j \leq m$.
Recall from \eqref{index2} the definition of the index sets $\Lambda_1(\leq m,n)$
and $\Lambda_\infty(\leq m,n)$.

\begin{theorem}
Let $\omega = (\log n)$. Then
\begin{equation*}
\boldsymbol{\lambda}\big(\mathcal{H}_\infty^{N_\infty(\leq m,n)}(\omega)\big) \sim (1 + \log m)^n\,,
\end{equation*}
and
for each $n$
\begin{align*}
\lim_{m \to \infty}
\frac{\boldsymbol{\lambda}\big(\mathcal{H}_\infty^{N_\infty(\leq m,n)}(\omega)\big)}
{\log^n m}=  \Big(\frac{4}{\pi^2}\Big)^n\,,
\end{align*}
where
\begin{align*}
N_\infty(\leq m,n)  = \big\{ k \in \mathbb{N} \colon
\text{$k = \mathfrak{p}^\alpha$, \,\, $\alpha \in \Lambda_\infty(\leq m,n)$ }
\big\}\,.
\end{align*}
\end{theorem}

\begin{proof}
It follows from  Proposition~\ref{H=H} and Example~\ref{examples2} together with  \eqref{bohrtrafo} that the Bohr map
\[
\mathcal{B} \colon  \mathcal{H}_\infty^{N_\infty(\leq m,n)}
(\omega)\to
\text{Trig}_{\Lambda_\infty(\leq m,n)}\big(\mathbb{T}^{n}\big)\,,\,\,\,\,\,\,
\sum_{n \in J} a_n e^{-\omega_n s} \mapsto \sum_{\alpha  \in\Lambda_\infty(\leq m,n)} a_{\mathfrak{p}^\alpha} z^\alpha
\]
defines an isometric linear bijection which preserves Bohr and Fourier coefficients. Hence the  first asymptotic is
a~consequence of Corollary~\ref{LoKha} (iv), and
the formula for the limit of  Corollary~\ref{limit formula}.
\end{proof}

Replacing the use of  Corollary~\ref{limit formula} by Corollary~\ref{manydimensionsB} and~\eqref{cardi},
similar arguments lead to the following counterpart of the preceding theorem.

\begin{theorem} \label{lessm}
Let $\omega = (\log n)$. Then for each $m,n\in \mathbb{N}$ one has
\begin{equation*}
\frac{1}{\sqrt{2^{m}}} \left(\sum_{k=0}^m \binom{n+k+1}{k}\right)^{\frac{1}{2}} \,\,\, \leq \,\,\,
\boldsymbol{\lambda}\big(\mathcal{H}^{N_1(\leq m,n)}_\infty(\omega)\big) \leq \left(\sum_{k=0}^m \binom{n+k+1}{k}\right)^{\frac{1}{2}}\,,
\end{equation*}
where
\begin{align*}
N_1(\le m,n)  = \big\{ k \in \mathbb{N} \colon
\text{$k = \mathfrak{p}^\alpha$, where $\alpha \in \Lambda_1(\leq m,n)$  }
\big\}\,.
\end{align*}
\end{theorem}

We finish this part with a result, which in view of Weissler's inequality from \eqref{weis} is  a formal extension of the last theorem. To do so, define for the finite index set $J \subset \mathbb{N}$, the numbers
\begin{equation*}
\text{ $\pi(J) = \max_{n \in J} \,\pi(n)$ \,\,\,\,\,\,and \,\,\,\,\,\, $\Omega(J) = \max_{n \in J} \,\Omega(n)$\,,}
\end{equation*}
and
\begin{equation*}
\Delta(J) = \big\{\alpha \in \mathbb{N}_0^{\pi(J)}\colon  n \in J\,, \,\, \mathfrak{p}^\alpha = n\big\}\,.
\end{equation*}
Here as usual $\pi(n)$ counts all primes $\leq n$ and $\Omega(n)$ is the number of prime divisors of $n$
counted according to their multiplicities.

\begin{theorem} \label{abstract}
Let $\omega = (\log n)$ and   $J \subset \mathbb{N}$ be a finite subset. Then
\begin{align*}
\boldsymbol{\lambda}\big(\mathcal{H}_\infty^{J}(\omega )\big)
= \int_{\mathbb{T}^{\pi(J)}} \Big| \sum_{\alpha \in \Delta(J)}z^\alpha\Big|\,dz\,.
\end{align*}
Moreover,
\begin{equation*}
\frac{1}{\sqrt{2^{\Omega(J)}}}\sqrt{|J|} \,\,\, \leq \,\,\,
\boldsymbol{\lambda}\big(\mathcal{H}_\infty^{J}(\omega )\big) \leq  \sqrt{|J|} .
\end{equation*}
\end{theorem}

\begin{proof}
As above (Proposition~\ref{H=H}), we see that the Bohr map
\begin{equation}\label{againBohr}
  \mathcal{B} \colon  \mathcal{H}_\infty^{J}(\omega) \to
\text{Trig}_{\Delta(J)}\big(\mathbb{T}^{\pi(J)}\big)\,,\,\,\,\,\,\,
\sum_{n \in J} a_n e^{-\omega_n s} \mapsto \sum_{\alpha  \in \Delta(J)} a_{\mathfrak{p}^\alpha} z^\alpha
\end{equation}
is isometric and  coefficient preserving. Again, the conclusion follows from Corollary~\ref{manydimensionsB}.
Indeed, the order $|\alpha|$ of each multi index  $\alpha \in \Delta(J)$ is at most $\Omega(J)$, and the cardinalities of $J$ and $\Delta(J)$ are the same.
\end{proof}

\subsubsection{{\bf Case III}}
Finally, we produce an example, where we mix aspects from the preceding two cases.
The following result considers ordinary Dirichlet polynomials of length $x$ which are  supported
on   all $1 \leq n=\mathfrak{p}^\alpha \leq x $ having  precisely $m$ prime devisors (all counted according to their multiplicities).

\begin{theorem} \label{fini}
Let $\omega = (\log n)$. Then for each $m \in \mathbb{N}$ one has
\begin{equation*}\label{m>1}
\boldsymbol{\lambda}\big(\mathcal{H}_\infty^{N_1(m,x)}(\omega)\big) \,\,\,\sim_{C(m)}\,\,\,
\sqrt{\frac{x}{\log x}}\,\,\,\sqrt{\frac{(\log\log x)^{m-1}}{(m-1)!}}\,,
\end{equation*}
where the constants  depend on $m$ but not on $x$, and
\[
N_1(m,x) = \big\{ 1 \leq  n \leq x \colon n = \mathfrak{p}^\alpha \,\, \,\,\, \text{with} \, \, \,\,\,
|\alpha| = m\big\}\,.
\]
Moreover,  the constant $C(m)$  is independent of $m$, whenever $m\le \frac{\log\log x}{2e}$.
\end{theorem}

Let us here first look at the special case $m=1$. Note first that for any $m$ by Proposition~\ref{H=H},
\[
\mathcal{H}_\infty^{N_1(m,x)}(\omega ) = \text{Trig}_{\Delta(N_1(m,x))}\big(\mathbb{T}^{\pi(N_1(m,x))}\big)\,,
\]
where the identification, given by the Bohr map, is isometric and coefficient preserving.
Then for $m=1$ we immediately deduce from Theorem~\ref{logp} and the prime number theorem that
\begin{align*} \label{m=1}
\lim_{x \to \infty}  \frac{\boldsymbol{\lambda}\big(\mathcal{H}_\infty^{N_1(1,x)}\big)}{\sqrt{\frac{x}{\log x}}}
=
\lim_{x \to \infty}  \frac{\boldsymbol{\lambda}\big(\mathcal{H}_\infty^{N_1(1,x)}\big)}{\sqrt{\pi{(x)}}}
= \frac{\sqrt{\pi}}{2}\,.
\end{align*}
So the preceding theorem extends this asymptotic estimate  from $m=1$ to arbitrary $m >1$
(neglecting constants).

\begin{proof}[Proof of Theorem~$\ref{fini}$]
Since $\pi(N(m,x)) \leq \pi(x)$ and $\Omega(N_1(m,x) )) = m$,  by Theorem~\ref{abstract}
\begin{equation*}
\frac{1}{\sqrt{2^m}}\sqrt{|N_1(m,x) |} \,\,\, \leq \,\,\,
\boldsymbol{\lambda}\big(\mathcal{H}_\infty^{\leq x, m } \big)\leq  \sqrt{|N_1(m,x)|}\,.
\end{equation*}
Then the   first claim follows from a well-known result of Landau (see, e.g., \cite[p.~200]{tenenbaum1995introduction})
showing that
\begin{equation*}\label{landau}
|N_1(m,x)| \,\,\,\sim_{C(m)}\,\,\, \frac{x}{\log x} \frac{(\log \log x)^{m-1}}{(m-1)!} 
 \quad\, \text{as \, $x \to \infty$}\,.
\end{equation*}
 For the proof of the second claim we use again that by Theorem~\ref{abstract}
\[
\boldsymbol{\lambda}\big(\mathcal{H}_\infty^{N_1(m,x)}\big) = \int_{\mathbb{T}^\infty} \Big| \sum_{\alpha \in \Delta(N_1(m,x))} z^\alpha\Big|\,dz\,.
\]
But in \cite[Equation (13),  p.107]{bondarenkoseip2016} it is shown that there is
a~universal constant $C \ge 1$ such that for any $x, m$ with $m\le \frac{\log\log x}{2e}$
\[
\int_{\mathbb{T}^\infty} \Big| \sum_{\alpha \in \Delta(N_1(m,x))} z^\alpha\Big|\,dz \,\,\,\sim_C\,\,\,
\Bigg(\int_{\mathbb{T}^\infty} \Big| \sum_{\alpha \in \Delta(N_1(m,x))} z^\alpha\Big|^2 dz\Bigg)^{\frac12} \, = \, \sqrt{|N_1(m,x)|}\,.
\]
On the other hand, applying the  Sathe-Selberg formula (see, e.g., \cite{ErdoesSarkozy}),
we conclude that for every $\delta>0$ there is $C(\delta) > 0$ such that for $m\le (2-\delta)\log\log x$
\[
|N_1(m,x)| \,\,\,\sim_{C(\delta)}\,\,\, \frac{x}{\log x} \frac{(\log \log x)^{m-1}}{(m-1)!} \quad\, \text{as \, $x \to \infty$}\,.
\]
This completes the argument.
\end{proof}

\section{Comparing Sidon constants}\label{sidon}

Recall that, given
a~topological group $G$, the Sidon constant $\boldsymbol{\text{\bf Sid}}(\Gamma)$ of a finite set  $\Gamma$ of characters in the dual group $\widehat{G}$ is given by the best constant
constant $c\ge 0$ such that for every trigonometric polynomial $f = \sum_{\gamma \in \Gamma} c_\gamma \gamma$ on $G$ one has
\begin{equation}\label{sidony}
\sum_{\gamma \in \Gamma} |c_\gamma| \leq c \|f\|_\infty\,.
\end{equation}
It is easily proved that  this constant  in fact equals
the unconditional basis constant  formed by all  $\gamma \in \Gamma$ in the Banach space $\text{Trig}_{\Gamma}(G)$.

Recall that the unconditional basis constant of a basis $(e_i)_{i \in I}$ of a Banach space $X$ is given by
the infimum over all $K > 0$ such that for any finitely supported family $(\alpha_i)_{i \in I}$ of scalars and for any finitely supported family $(\varepsilon_i)_{i \in I}$ with  $|\varepsilon_i| =1, \,  i \in I$ we have
\begin{equation*}\label{unconditionality}
\Big\Vert  \sum_{i \in I} \varepsilon_i \alpha_i e_i \Big\Vert \leq K \Big\Vert \sum_{i \in I} \alpha_i e_i \Big\Vert\,;
\end{equation*}
moreover, in this case the unconditional basis constant of $(e_i)_{i \in I}$ is defined to be the infimum over all these constants $K$.

In particular,  for any finite index set $\Gamma \subset \mathbb{Z}^n$ the unconditional basis constant
 of the collection of all monomials
$(z^\alpha)_{\alpha  \in \Gamma}$ in $\text{Trig}_{\Gamma}(\mathbb{T}^n)$, here denoted by $\boldsymbol{\chimon}\big(\text{Trig}_{\Gamma}(\mathbb{T}^n)\big)$,  equals the Sidon constant $\boldsymbol{\text{\bf Sid}}(\Gamma)$.

Given a finite dimensional Banach space $X$, the unconditional basis constant
with respect to a basis of this space and its projection constant are two quite different objects (compare for  example
${\chi} (\ell_2^n) =1$ with $\boldsymbol{\lambda} (\ell_2^n) \sim \sqrt{n}$).

But in the case of Banach spaces of Dirichlet polynomials and their associated spaces of multivariate trigonometric polynomials,
it turns out that a better understanding of one of the two constants often  leads to a better understanding of the other constant.

To illustrate  this point of view, we start considering  analytic trigonometric polynomials in one variable. Recall that  $\text{Trig}_{ \{k\colon 1 \leq k\leq d\}}(\mathbb{T})$ stands for all analytic trigonometric polynomials of the form $P(z) = \sum_{k=1}^{d} a_k z^k,\, z~\in~\mathbb{T}\,,$ so polynomials of degree $\leq d$ without a constant term $a_0$
(following our notation from Section~\ref{Integral formula - trigonometric polynomials}).  Then
 Rudin~\cite{rudin1959some} and
Shapiro \cite{shapiro1952extremal}
(see also \cite[Proposition~9.7]{defant2019libro}) proved that
\begin{equation*}\label{exo1}
   \frac{1}{\sqrt{2}} \sqrt{d}\leq \boldsymbol{\chimon} \big(\text{Trig}_{ \{k\colon 1 \leq k\leq d\}}(\mathbb{T})\big) \leq \sqrt{d}\,.
\end{equation*}
If we allow  constant terms, then the best known estimate is
\begin{equation*}\label{exo2}
   \sqrt{d}  - O( \log d)^{\frac{2}{3}+ \varepsilon} \leq \boldsymbol{\chimon} \big(\text{Trig}_{ \{k\colon 0 \leq k\leq d\}}(\mathbb{T})\big)
\leq \sqrt{d}\,.
\end{equation*}
This  is a~deep fact  proved by Bombieri and Bourgain in \cite{bombieri2004remark},
and it shows that at least  from the technical point of view
a~seemingly small perturbation of the index set may change the situation drastically.

Let us compare these results with what  we in Corollary~\ref{LoKha} proved for projection constants, namely
\[
\boldsymbol{\lambda}\big(\text{Trig}_{ \{k\colon 1 \leq k\leq d\}}(\mathbb{T})\big) \sim
\boldsymbol{\lambda}\big(\text{Trig}_{ \{k\colon 0 \leq k\leq d\}}(\mathbb{T})\big) \sim 1 + \log d\,.
\]
This illustrates  that unconditional basis constants (so Sidon constants) and projection constants of spaces of trigonometric
polynomials in one variable behave quite differently.

Of course, the situation  doesn't improve if we admit more variables -
nevertheless we may  recognize some in a sense  systematic  patterns.
Indeed, it was proved in \cite{defant2011bohnenblust} (a result elaborated in  \cite[Theorem 9.10]{defant2019libro})
 that
\[
\boldsymbol{\chimon}\left( \text{Trig}_{\leq m}(\mathbb{T}^n)\right) \sim_{C^m} \sqrt{\binom{n+m}{m}}.
\]
On the other hand  by Theorem~\ref{lessm} we know that
\[
\boldsymbol{\lambda}\left( \text{Trig}_{\leq m}(\mathbb{T}^n)\right) \sim_{C^m} \sqrt{\binom{n+m+1}{m}}\,,
\]
so in particular we have that
\begin{equation}\label{morgenU}
  \boldsymbol{\chimon}\left( \text{Trig}_{\leq m}(\mathbb{T}^n)\right)
\sim_{C^m}
\boldsymbol{\lambda}\left( \text{Trig}_{\leq m-1}(\mathbb{T}^n)\right) \sim_{C^m} \left( \frac{n+m}{m}\right)^{\frac{m-1}{2}}\,.
\end{equation}

Coming back to spaces of ordinary(!) Dirichlet polynomials, observe first that given a finite subset $J$ of $\mathbb{N}$,
it is obvious that the 'monomials' $n^{-s}, \, n \in J$ form a basis of  the Banach space $\mathcal{H}_\infty^{J}\big((\log n)\big)$,
in the following abbreviated by $\mathcal{H}_\infty^{J}$, and so we denote by
\[
\boldsymbol{\chimon}  \big(\mathcal{H}_\infty^{J} \big)
\]
the unconditional basis constant of all $n^{-s}, \, n \in J$.

As already used  in \eqref{againBohr}, there is an isometry
$
\mathcal{H}_\infty^{J} = \text{Trig}_{\Delta(J)}(\mathbb{T}^{\pi(J)})\,,
$
which  preserves  coefficients,
and hence by \eqref{morgenU} we conclude that
\[
\boldsymbol{\chimon}  \big(\mathcal{H}_\infty^{N_1(m,n)} \big)  \sim_{C^m}  \boldsymbol{\lambda}  \big(\mathcal{H}_\infty^{N_1(m-1,n)} \big)
\sim_{C^m} \sqrt{\binom{n+m}{m}}\,;
\]
here we again use the notation fixed in Theorem~\ref{lessm}.

What about the unconditional basis constants of $\mathcal{H}_\infty^{\leq x} $,
the Banach space of all Dirichlet polynomials
$D(s) = \sum_{n=1}^{x}a_n n^{-s}$ of length $x$, in comparison with their projection constants?

Again we use the notation from  Section~\ref{Case II}.
The following asymptotic
\begin{equation*} \label{annals}
\boldsymbol{\chimon} \big( \mathcal{H}_\infty^{\leq x} \big) =
\boldsymbol{\chimon} \big(\text{Trig}_{\Delta(x)}(\mathbb{T}^{\pi(x)}) \big)
\sim_C
\frac{\sqrt{x}}{e^{\big(\frac{1}{\sqrt{2}} + o(1)\big)
\sqrt{\log x \log \log x}}}\,.
\end{equation*}
is taken from \cite[Theorem~3]{defant2011bohnenblust} (see also \cite[Theorem~9.1]{defant2019libro}), and it is
the final outcome of a~long lasting  research project started by  Queff\'elec, Konyagin, and de la Bret\`eche
(see \cite[Section~9.3]{defant2019libro} for more details on its history). In Theorem~\ref{harpo} and
Corollary~\ref{harpo2} we proved  its  counterpart for projection constants:
\[
\boldsymbol{\lambda}\big( \mathcal{H}_\infty^{\leq x} \big)
= \boldsymbol{\lambda} \big(\text{Trig}_{\Delta(x)}(\mathbb{T}^{\pi(x)})\big)
\sim_C \frac{\sqrt{x}}{(\log \log x)^{\frac{1}{4}}}\,.
\]

In the $m$-homogeneous case analog results are known - by Balasubramanian-Calado-Queff\'elec \cite{balasubramanian2006bohr}
(a~result elaborated in \cite[Theorem~9.4]{defant2019libro}) one has (with the notation from Theorem~\ref{fini})
\[
\boldsymbol{{\chimon}} \big(\mathcal{H}_\infty^{N_1(m,x)} \big) = \boldsymbol{\chimon}\big(\text{Trig}_{\Delta(N_1(m,x))}\big(\mathbb{T}^{\pi(N_1(m,x))}\big) \big)
\sim_{C(m)} \frac{x^{\frac{m-1}{2m}}}{( \log x)^{\frac{m-1}{2}}}\,,
\]
whereas the corresponding result for projection constants from Theorem~\ref{fini} reads
\[
\boldsymbol{\lambda}\big(\mathcal{H}_\infty^{N_1(m,x)}\big) =
\boldsymbol{\lambda} \big(\text{Trig}_{\Delta(N_1(m,x))}\big(\mathbb{T}^{\pi(N_1(m,x))}\big) \big) \,\,\,\sim_{C(m)}\,\,\,
\sqrt{\frac{x}{\log x}}\sqrt{\frac{(\log\log x)^{m-1}}{(m-1)!}}\,.
\]

\section*{Acknowledgments}

We thank the referee for careful reading of the manuscript and valuable comments, which led to improvements in the presentation. On behalf of all authors, the corresponding author states that there is no conflict of interest.

\end{document}